\documentclass[11pt]{article}%

\usepackage[T1]{fontenc}
\usepackage{amsfonts, amssymb,amsmath, amscd,slashed,amsthm,mathtools}
\usepackage{subfigure}
\usepackage{hyperref}

\newcommand{\floor}[1]{\left\lfloor#1\right\rfloor}

\newcommand{\RR}{\mathbb{R}^2}
\newcommand{\dist}{\textnormal{dist}}

\newtheorem{theorem}{Theorem}[section]
\newtheorem{lemma}[theorem]{Lemma}%
\newtheorem{corollary}[theorem]{Corollary}%
\newtheorem{proposition}[theorem]{Proposition}%

\newtheorem{remark}[theorem]{Remark}%
\newtheorem{conjecture}[theorem]{Conjecture}%
\newtheorem{claim}[theorem]{Claim}%

\newtheorem{definition}[theorem]{Definition}%

\begin{document}

\title{Coloring distance graphs on the plane }

\author{Joanna Chybowska-Sok\'o\l{} \thanks{j.sokol@mini.pw.edu.pl}
\thanks{partially supported by the National Science Center of Poland under grant no.~2016/23/N/ST1/03181.} \and {Konstanty} {Junosza-Szaniawski}\thanks{konstanty.szaniawski@pw.edu.pl} \and {Krzysztof} {Węsek}\thanks{k.wesek@mini.pw.edu.pl}}

\date{{Faculty of Mathematics and Information Science}, {
Warsaw University of Technology}, {{Koszytkowa 75},  {00-662}  {Warsaw},{Poland}}}

\maketitle
\abstract{We consider the coloring of certain distance graphs on the Euclidean plane. Namely, we ask for the minimal number of colors needed to color all points of the plane in such a way that pairs of points at distance in the interval $[1,b]$ get different colors. 
The classic Hadwiger-Nelson problem is a special case of this question -- obtained by taking $b=1$.
The main results of the paper are improved lower and upper bounds on the number of colors for some values of $b$. In particular, we determine the minimal number of colors for two ranges of values of $b$ - one of which is enlarging an interval presented by Exoo and the second is completely new. Up to our knowledge, these are the only known families of distance graphs on $\mathbb{R}^2$ with a determined nontrivial chromatic number. Moreover, we present the first $8$-coloring for $b$ larger than values of $b$ for the known $7$-colorings.
As a byproduct, we give some bounds and exact values for bounded parts of the plane, specifically by coloring certain annuli.\\
keywords: coloring, distance graphs, Hadwiger-Nelson problem\\MSC Classification {05C15, 05C10, 05C62}}

\section{Introduction}\label{sec:intro}

How many colors are needed to color the Euclidean plane $\mathbb{R}^2$ so that no pair of points at distance $1$ get the same color?  This famous open question is known as the Hadwiger-Nelson problem, named after Hugo Hadwiger and Edward Nelson. It is also often formulated as a question about the chromatic number of a unit distance graph of the plane. A graph with a set of vertices $\mathbb{R}^2$ and a set of edges as pair of points at euclidean distance. Any of its subgraphs is called a \emph{unit distance graph}. For short, the parameter in question is often called \emph{the chromatic number of the plane}.

\begin{figure}[h]
\centering
		\includegraphics[width=0.2\textwidth]{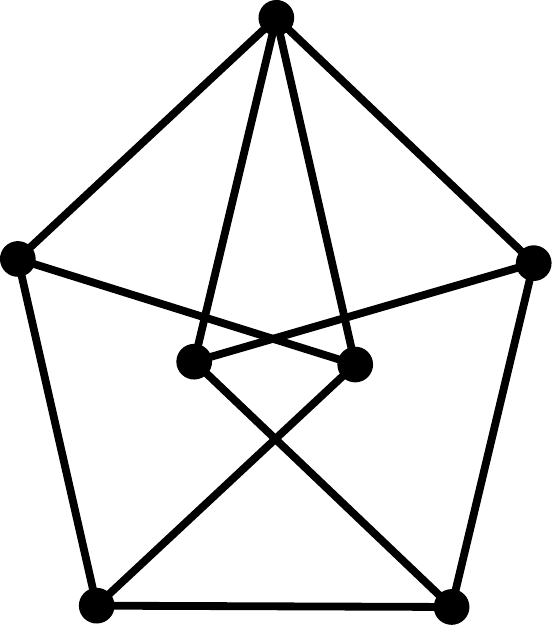}
		\caption{The so-called Moser spindle graph embedded as a unit distance graph in the plane (edges stand for segments of length $1$), with the chromatic number equal to $4$.}
		\label{fig:moser_spindle}
\end{figure}

\begin{figure}[h]
	\begin{center}
		\includegraphics[width=0.9\textwidth]{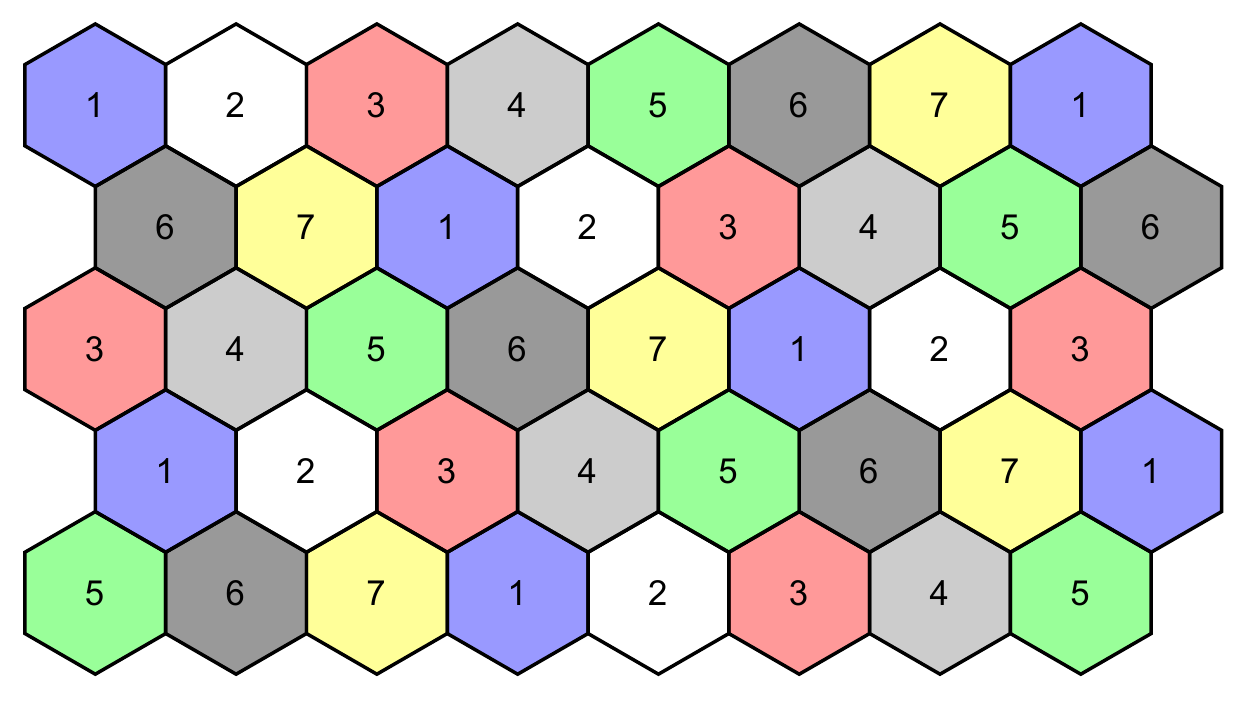}
		\caption{The coloring scheme of a of $7$-coloring of the unit distance graph of the plane, with hexagons of diameter slightly smaller than $1$ (sufficiently close to $1$).}
		\label{fig:7col-eps}
	\end{center}
\end{figure}

The problem was originally proposed by Edward Nelson in 1950 (see \cite{Soifer2009}) in hope of being helpful in the four-color problem of coloring maps. This idea did not work as desired, but the question proved itself to be interesting on its own, to put it mildly. In the same year, Nelson observed that at least $4$ colors are needed, as there are small, finite sets of points that are not $3$-colorable. The smallest example, consisting of just $7$ points, was found by Moser and Moser \cite{MoserMoser1961} and is known as Moser spindle (see Figure~\ref{fig:moser_spindle}). Still, in 1950, John Isbell found the upper bound of $7$ by the following coloring: take a tiling of the plane by regular hexagons of a diameter slightly smaller than $1$ and then color each hexagon with one color according to the scheme presented in Figure~\ref{fig:7col-eps} (colors of the borders do not matter). It is easy to check that two different hexagons of the same color are at distance greater than $1$, thus this coloring satisfies the required condition. The same coloring was considered a few years before by Hadwiger \cite{Hadwiger1945}, although in a different context. Yet Hadwiger \cite{Hadwiger1961} was the first to publish both bounds in a scientific article.
Somehow surprisingly, the aforementioned bounds remained unchanged for 68 years - as long as we consider the full generality. Nevertheless, advanced studies of the question, its subproblems, and other related topics provided some understanding.
For example, if we consider only measurable colorings (i.e. with measurable colors) then at least $5$ colors are necessary (Falconer \cite{Falconer1981}) and if we demand that the coloring consists of regions bounded by Jordan curves then at least $6$ colors are required (incorrect proof by Woodall \cite{WOODALL1973}; corrected proof by Townsend \cite{Townsend1979,Townsend2005}). On the other hand, it follows from De Bruijn–Erdos Compactness Theorem \cite{BruijnErdos1951} that the chromatic number of the plane is equal to the maximum chromatic number of its finite subsets, assuming the axiom of choice. Note that this assumption is crucial, as the influence of axiomatization of set theory on the chromatic number of geometrical graphs is a nuanced topic (see Soifer \cite{Soifer2009} for a discussion) with implications concerning the existence of measurable colorings. In particular, Payne \cite{Payne2009} constructed a unit distance graph with the chromatic number different for two different consistent axiom systems - the one with the axiom of choice giving the smaller number. In fact, in his considerations, the difference between the two axiomatizations essentially corresponds to considering measurable colorings and arbitrary colorings, respectively - hence showing that demanding measurability, in general, makes a difference.
However, we do not know whether the axiom of choice is relevant to the chromatic number of the plane itself.
Generally, across the decades, the Hadwiger-Nelson problem inspired many interesting results in combinatorics, geometry, topology, measure theory or abstract algebra, a vast number of challenging problems, and various applications. The list of variants include for example coloring of Euclidean spaces of higher dimensions 
\cite{ExooIsmailescuLim2014, ExooIsmailescu2014, CherkashinKulikovRaigorodskii2018}, fractional coloring 
\cite{ScheinermanUllman2011, GrytczukJunoszaSokolWesek2016, CranstonRabern2017, ExooIsmailescu2017} or circular coloring \cite{DeVosEbrahimiGheblehGoddynMoharNaserasr2007, JunoszaSzaniawski2018}.
We refer the reader to the article of Soifer \cite{Soifer2009} for an extensive discussion on the history of the question (including Soifer's private investigations) and a pleasant presentation of selected related problems.

And then the breakthrough happened. In 2018, Aubrey de Grey \cite{deGrey2018}, biogerontologist and computer scientist, proved that the chromatic number of the plane is at least $5$. Soon, a different, independent proof was published by Exoo and Ismailescu \cite{ExooIsmailescu2019}. Both first proofs were based on constructing a finite set of points (or a finite unit distance graph, if you prefer) forcing $5$ colors - although not as small and simple as the ones used to force $4$ colors. In both cases, the proof is a mixture of theoretical reasoning and computer computations. The breakthrough attracted a new wave of interest in the problem and its relatives. For example, a new Polymath project has been started \cite{Polymath16} in order to coordinate collaboration between those interested, both professional and amateur mathematicians. Considerable efforts were involved to the goal of finding as small example as possible (see Heule \cite{Heule2018_5, Heule2019}, Parts \cite{Parts2020_minimization}), in particular using clausal proof optimization. Note that the basic de Grey's graph consisted of $20425$ vertices, which was shrunk by the author to $1581$ by additional steps - but more minimization was possible. The current record of $509$ vertices is held by Parts \cite{Parts2020_minimization}. The minimization efforts form other examples supporting a general observation that the development in computer's computational power helped in the development of this area - the mentioned papers, apart from interesting ideas, make use of even hundreds or thousands of CPU computation hours. In some cases, a computer is used to run an algorithm written specifically for a certain subproblem, in other cases some general tools are used to check the colorability of constructed graphs, like SAT or Integer Programming solvers. Let us give an example of computer-driven progress related to a relative of the fractional chromatic number of the plane - without explaining the parameter here. In 2017, Cranston and Rabern \cite{CranstonRabern2017} published a paper with a clever proof by discharging method for the best known lower bound for this parameter. Currently, according to preliminary, unreviewed results of Polymath16 project (and thanks to a mixture of computers computation power and the power of the human mind), a unit distance graph on just $35$ vertices yielding a better lower bound has been found \cite{Parts2019_FCN_small} - and the best known bound is significantly better \cite{ExooIsmailescu2017,Parts2019_FCN_general}.
Only recently, a human verifiable (and still constructive) proof of de Grey's theorem was proposed by Parts \cite{Parts2020_5}, however, it still contains a large number of small cases for step-by-step checking.  The topic does not seem to be dried up, as another proof was presented by Voronov, Neopryatnaya, and Dergachev \cite{VoronovNeopryatnayaDergachev2021}. This construction has much more vertices but does not contain a copy of the Moser spindle, as opposed to previous constructions. In general, having various, especially relatively small non-$4$-colorable unit distance graphs with strong properties can be useful for another tempting step: hypothetic construction of a non-$5$-colorable unit distance graph, if it exists. First efforts in this direction already started, for example by Heule \cite{Heule2018_6}.

	Can we say something about the minimal size of such examples? It is not easy to obtain such assertions and the only known bounds are related to the following question: what is the maximal $6$-colorable (or $5$-colorable) portion of the plane? Results in this area were obtained by Pegg, Jr. \cite{Soifer2009}, which were improved by Pritikin \cite{Pritikin1998}, and then improved by Parts \cite{Parts2020_percent}. In particular, Parts' theorem states that more than $99.985698\%$ of the plane can be $6$-colored and at least $95.99\%$ of the plane can be $5$-colored. The presented colorings are then used to show that any subgraph of $G_{\{1\}}$ with at most $6992$ vertices can be $6$ colored, and any subgraph of $G_{\{1\}}$ with at most $24$ vertices can be $5$ colored.

The article is devoted to one of the most straightforward generalizations of Hadwiger-Nelson problem. In the classic question, only one distance is forbidden in any color class, namely $1$. What if we forbid more than one distance, let say, some set $D$ of distances? How many colors are needed for a coloring in which no color class contains a pair of points at distance from $D$? Our results concentrate on $D$ being an interval, but the general knowledge about other sets is also discussed.  
 
This leads to a more general notion of distance graphs. For  $D\subseteq\mathbb{R}_+$ and a metric space $X$   let us define \emph{the distance graph} $G_{D}(X)$ as a graph on the set of vertices $X$ and with $x,y\in X$ adjacent if the distance between $x$ and $y$ belongs to $D$. Distance graphs were considered first by Eggleton, Erdos and Skilton \cite{EggletonErdosSkilton1985,EggletonErdosSkilton1990} in case of $X=\mathbb{R}$ and $X=\mathbb{Z}$ (understood as Euclidean metric spaces). It would be difficult to give a comprehensive summary of research directions concerning the coloring of distance graphs, as various variants of distance graphs were already studied, including non-Euclidean spaces (see for example \cite{Kloeckner2015}). However, still relatively little is known for many problems considered so far. Particularly much work was devoted to integer distance graphs, that is, the case of $X=\mathbb{Z}$. 
For example, Katznelson \cite{Katznelson2001} and Ruzsa, Tuza, and Voigt \cite{RuzsaTuzaVoigt2002} independently proved that if a set $D$ consists of a sequence with exponential growth, then the chromatic number is finite.
On the other hand, providing a complete characterization of sets $D$ with finite $G_{D}(\mathbb{Z})$ seems to be a difficult problem connected to some highly non-trivial questions in additive number theory.
We note that some publications use the name 'distance graphs' for the class of integer distance graphs itself.

Here we are interested in $G_{D}(X)$ with $X=\mathbb{R}^2$ (understood as Euclidean metric space). For short, we will write $G_{D}$ for $G_{D}(\mathbb{R}^2)$. In the language of distance graphs, the unit distance graph of the plane can be described as $G_{\{1\}}$. Note that $G_{\{1\}}$ is isomorphic to $G_{\{d\}}$ for any $d>0$. What if the set $D$ contains more than one element? Let us first consider $\lvert D \rvert  =2$. As we can use scaling, without loss of generality we assume that $1$ is the smaller element of $D$. 
Probably, the first results concerning $\chi(G_{\{1,d\}})$ was proved by Huddleston \cite{OwingsTetivaHuddleston2008}. In this paper, few examples of $d$ forcing $\chi(G_{\{1,d\}})\ge 5$ were presented, but most importantly, it was shown that $\chi(G_{\{1,\frac{\sqrt{5}+1}{2}\}})\ge 6$ (by constructing of a finite set of points). Unaware of Huddleston's results, Katz, Krebs, and Shaheen \cite{KatzKrebsShaheen2014} independently proved one of Huddleston's lower bounds with $5$. Exoo and Ismailescu \cite{ExooIsmailescu2018} obtained more values forcing $5$ colors. From that point in time, 
$\chi(G_{\{1\}})\ge 5$ by de Grey comes into being and implies the mentioned lower bounds of $5$ colors. However, all those proofs are different and simpler than any proof of $\chi(G_{\{1\}})\ge 5$, and hence give some additional insight. Later on, Exoo and Ismailescu
\cite{ExooIsmailescu2020} showed that $\chi(G_{\{1,2\}})\ge 6$ (construction of a finite set of points, computer-aided proof); Palvolgyi and Agoston \cite{PalvolgyiAgoston2019} claim to have proved the same for $d=\sqrt{3}$ and $d=\frac{\sqrt{3}+1}{2}$ (worth noting: probabilistic method); Parts \cite{Parts2020_two_distances} presented a simpler proof of Exoo-Ismailescu result by constructing a set of only $31$ points. It remains open whether we can force $7$ colors with some $\{1,d\}$.


The opposite direction is to consider cases with $D$ being infinite and, moreover, unbounded. In particular, let us discuss $D$ equal to the set of odd integers. It was proved by Ardal, Manuch, Rosenfeld, Shelah and Stacho \cite{ArdalManuchRosenfeldShelahStacho2009} that $\chi(G_{\{1,3,\ldots\}})\ge 5$ (now implied by de Grey's result). However, the only known upper bound is the trivial $\aleph_0$, hence we do not even know if $\chi(G_{\{1,3,\ldots\}})$ is finite. On the other hand, it was proven by various authors that if we require colors to be Lebesgue measurable then an infinite number of colors is needed. The most recent proof is by Steinhardt \cite{Steinhardt2009} who used spectral graph theory.
This result can also be shown as a direct consequence of a theorem of Furstenberg, Katznelson, and Weiss \cite{FurstenbergKatznelsonWeiss1990}. For a measurable set $A\subseteq \mathbb{R}^2$ and Lebesgue measure denoted by $m$, the upper density of $A$ is defined as $\limsup\limits_{r\rightarrow +\infty} \frac{m(A\cap B(r))}{m(B(r))}$. The Furstenberg-Katznelson-Weiss theorem states that for $d\ge2$ and a measurable set $A\subseteq \mathbb{R}^d$ with positive upper density, all sufficiently large real numbers occur as distances between elements of $A$. The original proof of Furstenberg, Katznelson, and Weiss was ergodic-theoretic, see also Bourgain \cite{Bourgain1986} for a  harmonic-analytic proof and Falconer and Marstrand \cite{FalconerMarstrand1986} for a direct geometric proof.

Let us call a coloring of $\mathbb{R}^2$ \emph{measurable} if every single color is a Lebesgue measurable set. In~fact, Furstenberg-Katznelson-Weiss theorem implies a much more general statement: if $D$ contains arbitrarily large numbers then no measurable coloring of $G_D(\mathbb{R}^2)$ using a finite number of colors exists. 
Moreover, a theorem by Bukh \cite{Bukh2008} (which generalizes the Furstenberg-Katznelson-Weiss theorem) implies that even less strong assumption is sufficient. Namely, it is enough to assume that $D$ contains pairs of elements with arbitrarily large ratios (in particular, $D$ with elements arbitrarily close to $0$ satisfies this assumption). In other words, if $D$ is not a subset of an interval $[a,b]$ with $a,b\in \mathbb{R_+}$ then no measurable coloring using a finite number of colors exists.
On the other hand, if $D$ is a subset of such an interval then $\chi(G_D)$ is finite, as proved by Exoo \cite{Exoo2005}. 
However, we note that, for a fixed number of colors, nonexistence of a measurable coloring for a geometrically defined graph is not necessarily accompanied by nonexistence of any coloring. Shelah and Soifer \cite{ShelahSoifer2003,ShelahSoifer2004} and Soifer \cite{Soifer2005} presented examples of geometrically defined graphs that admit colorings using a finite number of colors but do not admit measurable colorings even for countably many colors. 
In particular, if we set $D=\{ \vert \sqrt{2}+q\rvert : q\in \mathbb{Q}\}$,  then $G_{D}(\mathbb{R})$ has a $2$-coloring but does not admit any measurable coloring with countably many colors. 

In this article, we are interested in $G_{D}$ for $D$ equal to an interval $[a,b]$ for some $0<a<b$. 
As we can use scaling, without loss of generality we assume that $a=1$.
Some important results on coloring of such graphs were presented by Exoo \cite{Exoo2005} (using slightly different notation). His motivation for considering such graphs was the following. 
Let us revisit the well known $7$-coloring of $G_{\{1\}}$ from Figure~\ref{fig:7col-eps}. This time take the same pattern of colors, but choose hexagons of diameter $1$. If we choose the colors of borders cleverly (see Figure~\ref{fig:7col_1}), then the coloring will still satisfy the condition for $G_{\{1\}}$. It turns out that this coloring is proper also for supergraphs of $G_{\{1\}}$ of the form of $G_{[1,b]}$ with $b\le \sqrt{7}/2$.
Hence, despite having more edges in such supergraphs of $G_{\{1\}}$, we still can bound the chromatic number by $7$. Can we obtain better lower bounds on $\chi(G_{[1,b]})$ for such $b$ than on $\chi(G_{\{1\}})$?

\begin{figure}[h]
	\begin{center}
		\includegraphics[width=0.2\textwidth]{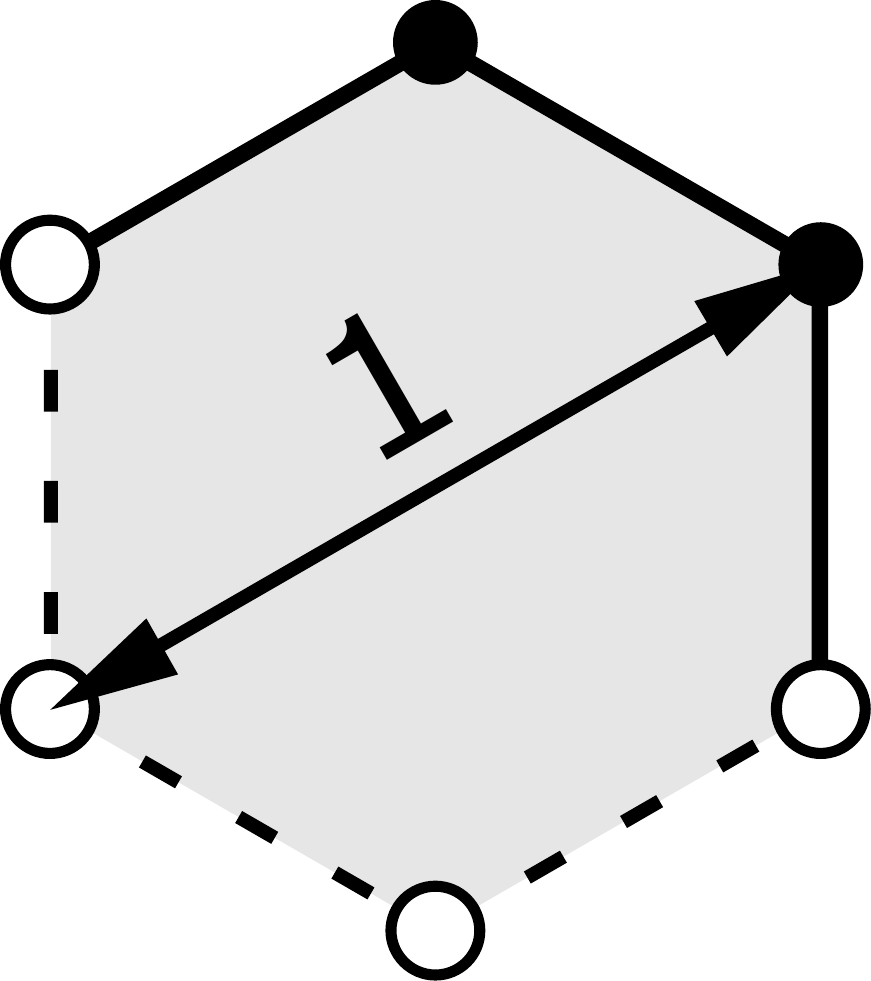}
		\caption{Color assignment from Figure \ref{fig:7col-eps} on the borders  (solid lines stand for the same color as in the interior). Proper for $G_{[1,b]}$ with $1 \le b\le \sqrt{7}/2$.}
		\label{fig:7col_1}
	\end{center}
\end{figure}

Among other results, Exoo managed to determine the chromatic number of some of such graphs $G_{[1,b]}$:
\begin{theorem}[Exoo \cite{Exoo2005}]\ \\ \label{thm:exoo}
	For $b\in (\sqrt{43}/5, \sqrt{7}/2] \approx (1.31149,1.32287]$ it holds $\chi(G_{[1,b]})=7$.
\end{theorem}

Up to now, the interval given in Theorem~\ref{thm:exoo} was the only known set of values of $b$ such that $\chi(G_{[1,b]})$ was determined. Moreover, to our knowledge, this family of distance graphs was the only family of distance graphs on $\mathbb{R}^2$ with a determined nontrivial chromatic number.
We emphasize that the real contribution of Theorem~\ref{thm:exoo} lays in establishing the lower bound for $\chi(G_{[1,b]})$ (the method behind it will be discussed later). 
Together with some computational experiments, Exoo considered Theorem~\ref{thm:exoo} a clue for a strong conjecture.
\begin{conjecture}[Exoo \cite{Exoo2005}]\ \label{conj:7-exoo} \\
	For $b>1$ sufficiently close to $1$, it holds $\chi(G_{[1,b]})=7$.
\end{conjecture}

Regarding values of $b$ that are closer to $1$ than in Theorem~\ref{thm:exoo}, Exoo provided the following:
\begin{theorem} [Exoo \cite{Exoo2005}]\ \\\label{thm:exoo_5}
	For $b\ge \frac{\sqrt{149}}{12} \approx 1.01721$ it holds $\chi(G_{[1,b]})\ge 5$.
\end{theorem}
Later, a strengthening of this statement to any $b>1$ was published.
\begin{theorem}[Grytczuk  \cite{GrytczukJunoszaSokolWesek2016}] \ \\ \label{thm:gryt_5}
	For $b>1$ it holds $\chi(G_{[1,b]})\ge 5$.
\end{theorem}%
However, it appears that years before the publication, Theorems~\ref{thm:exoo_5} and \ref{thm:gryt_5} were surpassed by a less known result obtained separately in two independent works.
In a series of papers by Brown, Dunfield, Perry \cite{DunfieldBrownPerryI, DunfieldBrownPerryII, DunfieldBrownPerryIII}, among other results, the authors gave an elegant proof by Dunfield that for any $b>1$ we have $\chi(G_{[1,b]})\ge 6$. The proof is using a result by Woodall (incorrect proof \cite{WOODALL1973}) and Townsend (correct proof based on a similar idea, see \cite{Soifer2009, Townsend1979, Townsend2005}).
Without giving the precise statement, the Woodall-Townsend theorem can be expressed in the following way: if the unit distance graph of the plane is colored with the condition that color classes are defined with Jordan curves, then at least $6$ colors are necessary. The key ingredient of the proof of the Woodall-Townsend theorem, very roughly speaking, is to find a point in the plane, which has at least $3$ colors in any $\varepsilon$-neighborhood. A related idea was used by Currie and Eggleton in their manuscript \cite{CurrieEggleton}, in which they (independently) prove the same result as Dunfield. Although the manuscript was not published, it was already mentioned by Currie in his other paper published in~1992 \cite{CURRIE1992}. Currie and Eggleton consider a proper coloring of $G_{[1,b]}$ and for $\varepsilon\in \big(0,(b-1)/2\big)$ find a point $x$ for which the closed $\varepsilon$-ball centered at $x$ contains at least $3$ colors. Then they prove that the annulus $\{p\in \mathbb{R}^2: 1+\varepsilon \le dist(p,x) \le b-\varepsilon\}$ needs at least $3$ colors and observe that it cannot use any of the colors from the closed $\varepsilon$-ball centered at $x$. This ends the proof. 
Therefore, let us state again the best known lower bound for $b$ close to $1$.
\begin{theorem}[Dunfield \cite{DunfieldBrownPerryI, DunfieldBrownPerryII, DunfieldBrownPerryIII}; Currie, Eggleton \cite{CurrieEggleton}] \ \\ 
	For $b>1$ it holds $\chi(G_{[1,b]})\ge 6$.
\label{thm:>=6}
\end{theorem}

We mentioned that De Bruijn–Erdos Compactness Theorem \cite{BruijnErdos1951} implies the chromatic number of the plane is witnessed by a finite subgraph, assuming the axiom of choice. This statement is true also for $G_{[1,b]}$ for any $b>1$. Krebs~\cite{Krebs2021} recently presented a construction of a finite subgraph of $G_{[1,b]}$ with chromatic number at least $5$ for any $b>1$. Although the result does not yield any new bound, the value lays in constructive nature combined with straightforward human verifiability.

A particular case of graphs $G_{[1,b]}$, namely $G_{[1,2]}$, has a practical motivation in telecommunication networks. Namely, subgraphs of $G_{[1,2]}$ can model hidden conflicts in a radio network (e.g., mobile phone network \cite{WalczakWojciechowski2006}). Coloring of $G_{[1,2]}$ produces a schedule that can solve such conflicts. For this motivation, also the fractional variant of graph coloring can be useful or even more efficient. Fractional coloring of graphs $G_{[1,b]}$ was also investigated\cite{GrytczukJunoszaSokolWesek2016}.

\subsection{Our approach}

It seems that the idea of a point close to at least $3$ colors was not exploited for larger values of $b$. 
In our work, we use this concept to provide new lower bounds for $\chi(G_{[1,b]})$ for certain values of $b>1$. The approach consists of two steps. 
First, we use the mentioned fact that any proper coloring of $G_{[1,b]}$ for any $b>1$ and any sufficiently small $\varepsilon>0$ admits a closed $\varepsilon$-ball centered at some point $x$ containing at least $3$ colors. We give a new proof of this statement. Without loss of generality, we can assume that $x=(0,0)$. As the second step, we consider the annulus $A_{b,\varepsilon}$ centered at $(0,0)$ with the inner radius $1+\varepsilon$ and the outer radius $b-\varepsilon$. Clearly, none of at least $3$ colors found in the closed $\varepsilon$-ball centered at $(0,0)$ can be used in $A_{b,\varepsilon}$. If for some $k$ we are able to prove that $A_{b,\varepsilon}$ itself requires at least $k$ colors, then we obtain $\chi(G_{[1,b]})\ge k+3$.

In order to show a lower bound for proper coloring of $G_{[1,b]}$ or a subgraph of $G_{[1,b]}$, one may try to construct a finite set of points for which finite graph coloring techniques can be applied. Many mentioned lower bounds, including de Grey's result, fit this scheme.
For example Exoo, in order to prove the lower bound in Theorem~\ref{thm:exoo}, considered proper coloring of a set $P$ of vertices of a bounded part a~carefully chosen regular triangular grid. 
Using computer-aided calculations, he showed that for the specified range of $b$ the subgraph of $G_{[1,b]}$ induced by $P$ requires at least $7$ colors. However, it is unlikely that his choice of parameters for the grid is optimal (in terms of the range of $b$), as it is limited by the computer computational power.
In order to show a lower bound for proper colorings of $A_{b,\varepsilon}$, we also construct a certain finite subset of it. Our analysis suggested that it is reasonable to consider sets created by taking a number of points regularly placed on a small number of circles of radius chosen between $1+\varepsilon$ and $b-\varepsilon$. 

The benefit of this approach is that we reduce the search for finite configurations to a relatively small part of the plane. 
On the other hand, it is likely that for many values of $b$ 
the~chromatic number of the~subgraph of $G_{[1,b]}$ induced by $A_{b,\varepsilon}$ plus $3$
is strictly smaller than 
$\chi(G_{[1,b]})$. However, this plan proves itself to be effective in providing a new contribution, as we were able to determine $\chi(G_{[1,b]})$ for two intervals of values of $b$.
Namely, we enlarge the interval given in Theorem~\ref{thm:exoo} (by providing a more general lower bound) and present a completely new interval for which $9$ colors are optimal. The results are presented in Section~\ref{sec:lower}.
Our lower bounds on the number of colors for finite configurations are obtained by computer-based computations.
We used one of the standard integer linear programming formulations of graph coloring.

As a byproduct of the method from Section~\ref{sec:lower}, a natural question arises: what is the~chromatic number of the~subgraph of $G_{[1,b]}$ induced by $A_{b,\varepsilon}$? Since Section~\ref{sec:lower} makes use of lower bounds, can we say something about upper bounds for such graphs? The topic of coloring bounded parts of the plane is discussed in Section~\ref{sec:annuli}.

Moreover, in Section~\ref{sec:upper} we present colorings of distance graph $G_{[1,b]}$ for various values of $b$ establishing some upper bounds. In particular we present coloring with $8$-colors for $b$ slightly bigger than $\sqrt{7}/2$. We also present a scheme for larger values of parameter $b$ that generalize and improves colorings based on hexagon tiling by Exoo \cite{Exoo2005} and Lonc \cite{lonc}.

\section{Preliminaries}
\label{sec:preliminaries}

First, we shall give some fundamental graph theoretical definitions. A \emph{graph} is a pair $G=(V,E)$ where $V$ is an arbitrary set and $E\subseteq \{\{x,y\}: x,y\in V\}$. Elements of $V$ are called \emph{vertices} and elements of $E$ are called \emph{edges}. For short, we often write $xy$ for an edge $\{x,y\}$. If $xy \in E$, then we say that $x$ and $y$ are \emph{adjacent} in $G$. In this thesis, we will consider graphs with finite and infinite sets of vertices. 
\begin{definition}
A \emph{coloring} of a graph $G=(V,E)$ is a function $c: V\rightarrow K$ (where $K$ is an arbitrary set of \emph{colors}) such that any $xy \in E$ satisfies $c(x)\neq c(y)$. We say $c$ is a \emph{proper $k$-coloring} if $\lvert K\rvert=k$, for finite $k$.
For a (non necessarily finite) graph $G$, \emph{the chromatic number of $G$} is defined by $$\chi(G) = \inf \{\lvert K\lvert : \text{a proper coloring of }G\textrm{ into }K\text{ exists}\}.$$ Note that for a finite graph, it is the minimal number of colors for a proper coloring.
\end{definition}

\begin{definition}
A graph $G_{[a,b]}$ is a graph whose vertices are all the point of the plane $V=\mathbb{R}^2$, in which two point are adjacent if their distance $d$ satisfies $a \leq d \leq b$.\\ 
 $G_{[a,b]}=(\mathbb{R}^2, \{ \{x,y\}\subset \mathbb{R}^2\ :\ a\leq \dist(x,y)\leq b  \}$
\end{definition}

For $b>1$ and $\varepsilon\ge 0$, let us formally define a class of annuli that we will use $$A_{b,\varepsilon}=\{p\in \mathbb{R}^2: 1+\varepsilon \le dist(p,(0,0)) \le b-\varepsilon\}.$$

\section{Lower bounds for the chromatic number of $G_{[1,b]}$}
\label{sec:lower}

We start with a key lemma already proved in \cite{CurrieEggleton}. However, we give a different, shorter proof of~this~fact (while the core idea is similar). For $\varepsilon>0$, by a \emph{closed $\varepsilon$-ball} centered in a point $x$ we understand the set $\{p\in \mathbb{R}^2: dist(p,x)\le \varepsilon\}$.

\begin{lemma}[\hspace{-0.1cm} \cite{CurrieEggleton}; a different proof] \label{lem:3_colors_point} \ \\
	Let $c$ be a proper coloring of $G_{[1,b]}$ for $b>1$. Consider any $\varepsilon$ satisfying $0<\varepsilon<b-1$. Then there exists a point $x$ in $\mathbb{R}^2$ such that in the closed $\varepsilon$-ball centered in $x$ there are at least $3$ colors (with respect to $c$).
\end{lemma}
\begin{proof}
	Take a proper coloring $c$ of $G_{[1,b]}$ for $b>1$ and fix $\varepsilon$ satisfying $0<\varepsilon<b-1$.
	For any nonempty monochromatic set $A \subseteq \mathbb{R}^2$, denote by $S(A)$ the following set: 
	with $c_1$ being the color~of~$A$, take the~set of all $c_1$-colored points of $\mathbb{R}^2$ that can be obtained from $A$ by a sequence of $c_1$-colored points with consecutive distances at most $\varepsilon$. 
	For $S \subseteq\mathbb{R}^2$, let $H_{\varepsilon}(S) = \{p\in \mathbb{R}^2: \exists_{s \in S} \; \; dist(p,s)\le \varepsilon\}$. In the proof, we will use the following observation.
	\begin{description}
		\item[($\ast$)] If $S$ is a bounded, connected set, then there exists a simple closed curve $C$ in $H_{\varepsilon}(S)\setminus S$ so that all points from $S$ are inside $C$.
	\end{description}
	
	Suppose the contrary to the thesis, that there is no point $x$ as in the lemma formulation. Take any point $y\in \mathbb{R}^2$. Set $S_1=S(\{y\})$. Note that $S_1$ is a bounded set. Otherwise, it would contain a sequence of points of the same color with consecutive distances less than $\varepsilon$ and realizing an arbitrarily large distance. Hence $S_1$ would contain a pair of points at distance in $[1,b]$, a contradiction. 
	Since 
	$H_{\varepsilon/2}(S_1)$ is bounded and connected, by ($\ast$) we can find a simple closed curve $C_1$ in $H_{\varepsilon}(S_1)\setminus H_{\varepsilon/2}(S_1) = H_{\varepsilon/2}(H_{\varepsilon/2}(S_1)) \setminus H_{\varepsilon/2}(S_1)$ so that all points from $H_{\varepsilon/2}(S_1)$ are inside $C_1$. 
	Let us show that all points of $C_1$ have the same color.
	Suppose that there are at least $2$ colors in $C_1$. Then there exists a pair of points at distance smaller than $\varepsilon/2$. Since points of $C_1$ are $\varepsilon/2$-close to a vertex colored with color of $y$ (distinct from colors in $C_1$), hence we have a closed $\varepsilon$-ball with $3$ colors, a contradiction. Thus only one color is used in $C_1$.
	We will iteratively construct $S_i,C_i$ for $i>1$ in the following way.
	Take $i>1$ and set $S_i=S(C_{i-1})$. 
	Since 
	$H_{\varepsilon/2}(S_i)$ is bounded and connected, by ($\ast$) we can find a simple closed curve $C_i$ in $H_{\varepsilon}(S_i)\setminus H_{\varepsilon/2}(S_i) = H_{\varepsilon/2}(H_{\varepsilon/2}(S_i)) \setminus H_{\varepsilon/2}(S_i)$  
	so that all points from $H_{\varepsilon/2}(S_i)$ are inside $C_i$. 
	Again, all points of $C_i$ have the same color, as otherwise, we would have a closed $\varepsilon$-ball with $3$ colors.
	
	
	We claim that $diam(C_i)-diam(C_{i-1})\ge \varepsilon$ for $i>1$. For a simple closed curve $C$, let $In(C)$ stand for the bounded component of $\mathbb{R}^2\setminus C$. Consider two points $y_1,y_2$ that realize $diam(C_{i-1})$ and take the line $\ell$ containing $y_1,y_2$. Let $y_1', y_2'$ be the points from $\ell$ satisfying $dist(y_1',y_1)=\varepsilon/2$, $dist(y_2',y_2)=\varepsilon/2$ and $dist(y_1',y_2')=dist(y_1,y_2)+\varepsilon$. Clearly, $y_1', y_2' \in H_{\varepsilon/2}(S_i,X_i) \subseteq In(C_{i})$ and hence $diam(C_i) \ge dist(y_1',y_2')$, as claimed. Thus $diam(C_i)-diam(C_{i-1})\ge \varepsilon$.
	Therefore for sufficiently large $i$ we have $diam(C_{i})>1$ and there are two points in $C_{i}$ at distance from $[1,b]$. On the other hand, all points $C_{i}$ have the same color, which contradicts with the fact that $c$ is a proper coloring of $G_{[1,b]}$.

\end{proof}


The proof of the following Theorem is partially computer aided. We use mixed integer programming solver to compute the coloring of some graphs. Coloring is modeled in a standard way: for a given graph $G$ with $V(G)=\{1,\ldots,n\}$ and a number of colors $K$. We want to check if $G$ is $K$-colorable. It is equivalent to the feasibility of the following integer linear program. For $i\in 1,\ldots,n$ and $k\in\{1,\ldots, K\}$, we introduce a binary variable $x_{i,k}$ indicating if vertex $i$ receives color $k$. The problem itself is a decision problem and there is no inherent objective function. However, it is useful to introduce an objective function using additional variables in order to break some symmetries of the model (see for example \cite{WilliamsYan2001}). Hence for $k\in\{1,\ldots, K\}$, we introduce a binary variable $y_k$ indicating if color $k$ is used.

\begin{equation*}
\begin{array}{ll@{}ll}
\text{minimize}  & & \displaystyle\sum\limits_{k=1}^{K} k \cdot y_k &\\
\text{subject to}& & \displaystyle\sum\limits_{k=1}^{K}   x_{i,k} \geq 1,  &i=1 ,\ldots, n\\
&   &x_{i,k} + x_{j,k} \leq 1,  &(i,j)\in E(G), \; k=1,\ldots,K\\
&   &x_{i,k} - y_k  \leq 0,  &i=1 ,\ldots, n, \; k=1,\ldots,K\\
&                                                &x_{i,k}, \: y_k \in \{0,1\}, &i=1 ,\ldots, n, \; k=1,\ldots,K
\end{array}
\end{equation*}

Now we are ready for the proof of the main theorem of this chapter.
\begin{theorem}\label{thm:main}
	The following inequalities hold:
	\begin{enumerate}
		\item $\chi(G_{[1,b]})\ge 7$ for $b> \sqrt{2-2 \sin(\frac{18 \pi}{325})} \approx 1.28599$
		\item $\chi(G_{[1,b]})\ge 8$ for $b> \sqrt{2+2\sin(\frac{\pi}{38})}  \approx 1.47145$
		\item $\chi(G_{[1,b]})\ge 9$ for $b> \sqrt{2+2\sin(\frac{7\pi}{45})} \approx 1.71433$
		\item $\chi(G_{[1,b]})\ge 10$ for $b> 2\sqrt{2}-1 \approx 1.82843$
		\item $\chi(G_{[1,b]})\ge 11$ for $b> \frac{1}{3}(5 - \sqrt{2} + \sqrt{6}) \approx 2.01176$
	\end{enumerate}
\end{theorem}
\begin{proof} 
	
	A key step of the proof is encapsulated in the following claim:
	\begin{claim}\label{claim:k+3}
		Let $b>1$. If $G_{[1,b]}[A_{b,\varepsilon}]$ for some $\varepsilon>0$ requires at least $k$ colors, then $\chi(G_{[1,b]})\ge k+3$.
	\end{claim}
	Let us prove Claim~\ref{claim:k+3}. Consider $b>1$ and a proper coloring $c$ of $G_{[1,b]}$. Fix $\varepsilon>0$. Let $x$ be a point obtained from Lemma~\ref{lem:3_colors_point}: such that in the closed $\varepsilon$-ball centered at $x$ there are at least $3$ colors with respect to $c$, say colors $1,2,3$. 
	Without loss of generality, we can assume $x=(0,0)$, as the coloring can be shifted.
	No point in $A_{b,\varepsilon}$ can be colored with any of the colors $1,2,3$. Hence $c$ uses at least $k+3$ colors. It follows that $\chi(G_{[1,b]})\ge k+3$, which concludes the proof of Claim~\ref{claim:k+3}.
	
	In order to make use of Claim~\ref{claim:k+3}, in each case, we constructed a finite subset of $A_{b,\varepsilon}$ such that, for sufficiently small $\varepsilon>0$, it induces a graph requiring at least $k$ colors in $G_{[1,b]}$ (for some $k$). In other words, we found a subgraph of $G_{[1,b]}$ consisting of vertices from $A_{b,\varepsilon}$ forcing $k$ colors. In each case, we formulated the coloring problem as a mixed integer programming instance and used a computer to show infeasibility for any number of colors smaller than the presented bound. Denote by $X^n_r$ the set consisting of $n$ points evenly distributed on the circle with the center in $x$ and radius $r$ so that one of the points lays on the upward vertical half-line from $x$.
	
%
	\begin{enumerate}
		\item Assume that $b>\sqrt{2-2 \sin(\frac{18 \pi}{325})}$. Consider $Y_{b,\varepsilon} = X^{1300}_{1+\varepsilon}\cup X^{1300}_{b-\varepsilon} \subseteq A_{b,\varepsilon}$. We checked that for sufficiently small $\varepsilon>0$ any proper coloring of the graph induced by $Y_\varepsilon$ in $G_{[1,b]}$ requires at least $4$ colors.
		
		\item Assume that $b>\sqrt{2+2\sin(\frac{\pi}{38})}$. Consider $Y_{b,\varepsilon} = X^{190}_{1+\varepsilon}\cup X^{190}_{b-\varepsilon} \subseteq A_{b,\varepsilon}$. We checked that for sufficiently small $\varepsilon>0$ any proper coloring of the graph induced by $Y_\varepsilon$ in $G_{[1,b]}$ requires at least $5$ colors.
		
		\item Assume that $b>\sqrt{2+2\sin(\frac{7\pi}{45})}$. Consider $Y_{b,\varepsilon} = X^{180}_{1+\varepsilon}\cup X^{180}_{(1+b)/2}\cup X^{180}_{b-\varepsilon} \subseteq A_{b,\varepsilon}$. We checked that for sufficiently small $\varepsilon>0$ any proper coloring of the graph induced by $Y_\varepsilon$ in $G_{[1,b]}$ requires at least $6$ colors.
		
		\item Assume that $b>2\sqrt{2}-1$. Consider $Y_{b,\varepsilon} = X^{120}_{1+\varepsilon}\cup X^{120}_{(1+b)/2}\cup X^{120}_{b-\varepsilon} \subseteq A_{b,\varepsilon}$. We checked that for sufficiently small $\varepsilon>0$ any proper coloring of the graph induced by $Y_\varepsilon$ in $G_{[1,b]}$ requires at least $7$ colors.
		
		\item Assume that $b> \frac{1}{3}(5 - \sqrt{2} + \sqrt{6})$. Consider $Y_{b,\varepsilon} = X^{120}_{1+\varepsilon}\cup X^{120}_{(1+b)/2}\cup X^{120}_{b-\varepsilon} \subseteq A_{b,\varepsilon}$. We checked that for sufficiently small $\varepsilon>0$ any proper coloring of the graph induced by $Y_\varepsilon$ in $G_{[1,b]}$ requires at least $8$ colors.
	\end{enumerate}

	As promised before, Claim~\ref{lem:3_colors_point} concludes the proof in each case.
\end{proof}

The exact right-hand side values in the inequalities on $b$ in Theorem~\ref{thm:main} are the optimal values for which the given finite configurations of points possess the desired chromatic properties. That is, in each case for any smaller value of $b$ and any small value of $\varepsilon>0$ (only small values of $\varepsilon>0$ were checked), the given set $Y_{b,\varepsilon}$ can be colored in $G_{[1,b]}$ with fewer colors than stated - hence we cannot use Claim~\ref{lem:3_colors_point} for the same lower bound. Nevertheless, we are far from claiming optimality of the given sets. In Theorem~\ref{thm:main} we simply present the best constructions that we were able to find and verify. 
We expect that there exist sets of similar forms which work for smaller values~of~$b$ in respective cases. We will shed some light on the range of possible improvements later.



By combining Theorem~\ref{thm:main} with previously known bounds, we can obtain two intervals of values of $b$ for which the chromatic number can be determined. Namely, let us use that Exoo \cite{Exoo2005} observed that $\chi(G_{[1,b]})\le 7$ for $b\le\sqrt{7}/2$ and Ivanov \cite{Ivanov2007} showed that $\chi(G_{[1,b]})\le 9$ for $b\le\sqrt{3}$ (see Figure~\ref{fig:9col}).

\begin{figure}[h]
	\begin{center}
		\includegraphics[width=0.8\textwidth]{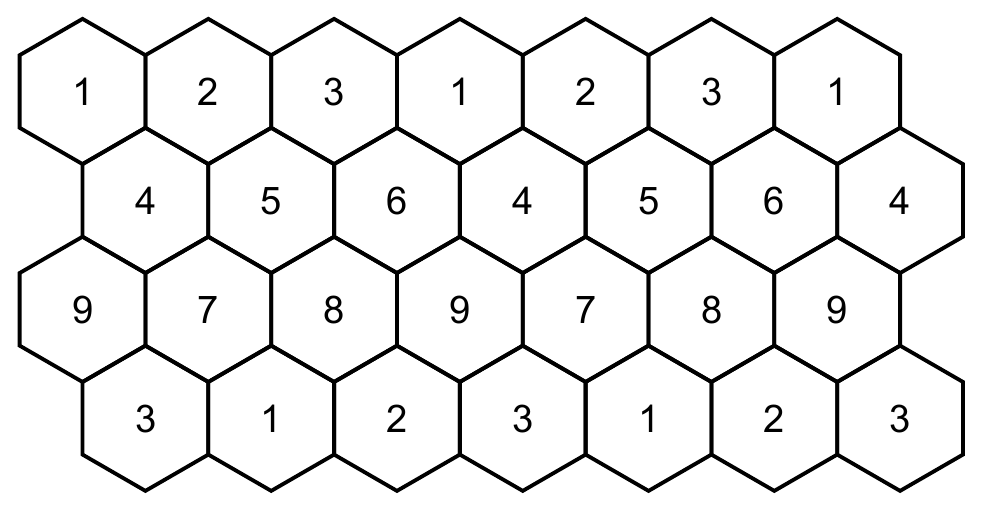}
		\caption{The coloring scheme for a proper $9$-coloring of the plane based on hexagons of diameter~$1$, by Ivanov \cite{Ivanov2007}. Color assignment on the borders according to the one presented in Figure~\ref{fig:7col_1}.
		}
		\label{fig:9col}
	\end{center}
\end{figure}

\begin{corollary}\label{cor:determined}\ \vspace{0pt}
	\begin{enumerate}
		\item 	For $b \in \Big(\sqrt{2-2 \sin(\frac{18 \pi}{325})}, \sqrt{7}/2 \Big] \approx (1.28599,1.32287]$ it holds $\chi(G_{[1,b]}) = 7$.
		\item 	For $b \in \Big(\sqrt{2+2\sin(\frac{7\pi}{45})}, \sqrt{3} \Big] \approx (1.71433, 1.73205]$ it holds $\chi(G_{[1,b]}) = 9$.
	\end{enumerate}
\end{corollary}

Note that the first interval contains and substantially enlarges the interval obtained by Exoo in Theorem~\ref{thm:exoo}. Moreover, no interval for which the chromatic number is $9$ was known before. Up to our knowledge, Corollorary~\ref{cor:determined} covers all the known distance graphs on $\mathbb{R}^2$ with a determined nontrivial chromatic number.

\section{Colorings of annuli}
\label{sec:annuli}
The proof of Theorem~\ref{thm:main} relies on lower bounds on the number of colors needed for $G_{[1,b]}[A_{b,\varepsilon}]$ for certain values of $b$ and sufficiently small $\varepsilon>0$. One may ask if we can give an upper bound for the chromatic number of these graphs (or $G_{[1,b]}[A_{b,0}]$, for simplicity). Using this knowledge we would be able to say something about the range of possible improvements in the method used in the proof of Theorem~\ref{thm:main}.

The chromatic number of subgraphs of $G_{\{1\}}$ induced by some subsets of the plane have been already studied. Bauslaugh \cite{Bauslaugh1998}, Perz \cite{Perz2018}, Bock \cite{Bock2019}, Oostema, Martins and Heule \cite{OostemaMartinsHeule2020} analyzed infinite strips, Clyde Kruskal \cite{Kruskal2008} considered circles, squares and regular polygons, Axenovich, Choi, Lastrina,  McKay,  Smith and Stanton \cite{AxenovichChoiLastrinaMcKaySmithStanton2014} studied subsets of $\mathbb{Q}\times\mathbb{R}$, set of vertices of a convex polygon, unions of lines and infinite strips.
The paper closest to our interest was published by Alm and Manske \cite{AlmManske2014}.
The authors considered particular classes of annuli, especially with respect to a special class of colorings called radial colorings. A coloring of an annulus $A$ is a \emph{radial coloring} if there exists a sequence of radii $r_1,\ldots,r_{k+1}=r_1$ so that the sector strictly between radii $r_i$ and $r_{i+1}$ is colored with a single color for $i\in\{1,\ldots,k\}$. Thus, the color of a point (except a finite number of radii) depends only on the angle of the half-line from the center of the annulus to this point. See Figure~\ref{fig:radial1} for an example of a radial coloring. 
The results of Alm and Manske do not have direct application for the annuli of interest in this section, as they studied annuli of outer radius smaller than $1$ (and as a subgraph of $G_{\{1\}}$). However, it appears that radial colorings themselves can be of use for our purposes.

We are interested in upper bounds for the chromatic number of $G_{[1,b]}[A_{b,\varepsilon}]$. It is sufficient to provide upper bounds for $\varepsilon=0$, since $\chi(G_{[1,b]}[A_{b,\varepsilon}]) \le \chi(G_{[1,b]}[A_{b,0}])$ for any $b>1, \varepsilon>0$.
Let us denote $A_{b}=A_{b,0}$. In the next result, we consider chromatic number of $A_{b}$ as a subgraph of $G_{[1,b]}$.
We present upper bounds for the chromatic number of annuli obtained by constructing radial proper colorings and combine them with lower bounds coming from the proof of Theorem~\ref{thm:main}.

\begin{proposition}\label{prop:col_of_annuli}
	Table~\ref{tab:annulus} presents bounds on $\chi(G_{[1,b]}[A_{b}])$.
	\begin{center}
		\begin{table}[h]
			\begin{tabular}{ | c  c  c | c | c }
				\hline
				$b\in$ & & & $\chi(G_{[1,b]}[A_{b}]) = $ \\ \hline
				
				$\Big(1,\sqrt{2 - 2 \sin(\frac{\pi}{18})}\Big]$ & $\approx$ & $(1,1.28558]$ & $3$ \\ \hline
				
				$\Big(\sqrt{2 - 2 \sin(\frac{\pi}{18})}, \sqrt{2-2 \sin(\frac{18 \pi}{325})}\Big]$ & $\approx$ & $(1.28558,1.28599]$ & $3$ or $4$ \\ \hline
				
				$\Big(\sqrt{2-2 \sin(\frac{18 \pi}{325})}, \sqrt{2} \Big]$ & $\approx$ & $(1.28599, 1.41421]$ & $4$ \\ \hline
				
				$\Big(\sqrt{2}, \sqrt{2+2\sin(\frac{\pi}{38})} \Big]$ & $\approx$ & $(1.41421, 1.47145]$ & $4$ or $5$ \\ \hline
				
				$\Big(\sqrt{2+2\sin(\frac{\pi}{38})}, \sqrt{\frac{3}{2} - \frac{\sqrt{5}}{2}} \Big]$ & $\approx$ & $(1.47145, 1.61803]$ & $5$ \\ \hline
				
				$\Big(\sqrt{\frac{3}{2} - \frac{\sqrt{5}}{2}} , \sqrt{2+2\sin(\frac{7\pi}{45})} \Big]$ & $\approx$ & $(1.61803,1.71433]$ & $5$ or $6$ \\ \hline
				
				$\Big(\sqrt{2+2\sin(\frac{7\pi}{45})}, \sqrt{3} \Big]$ & $\approx$ & $(1.71433, 1.73205]$ & $6$ \\ \hline
				
				$\Big(\sqrt{3}, 2 \cos(\frac{\pi}{7}) \Big]$ & $\approx$ & $(1.73205, 1.80194]$ & $6$ or $7$ \\ \hline
				
				$\Big(2 \cos(\frac{\pi}{7}),2\sqrt{2}-1 \Big]$ & $\approx$ & $(1.80194, 1.82843]$ & $6$ or $7$ or $8$ \\ \hline
				
				$\Big(2\sqrt{2}-1, \sqrt{2+\sqrt{2}} \Big]$ & $\approx$ & $(1.82843, 1.84776]$ & $7$ or $8$ \\ \hline	
			\end{tabular}
			\caption{Known bounds on $\chi(G_{[1,b]}[A_{b}])$ depending on $b$.}
			\label{tab:annulus}
		\end{table}
	\end{center}
\end{proposition}
\begin{proof}\ 
	All the presented lower bounds are implied by the proof of Theorem~\ref{thm:main}, as $\chi(G_{[1,b]}[A_{b,\varepsilon}]) \le \chi(G_{[1,b]}[A_{b}])$ for any $b>1, \varepsilon>0$. We need to prove the upper bounds. We will give detailed reasoning only for two of the upper bounds. In other cases, we only present the proper coloring - the rest follows similarly.
	\begin{enumerate}
		\item 
		We will show that $\chi(G_{[1,b]}[A_{b}]) \le 3$ for $b=\sqrt{2 - 2 \sin(\frac{\pi}{18})}$ .
		
		We shall define a radial proper $3$-coloring $c$ of $G_{[1,b]}[A_{b}]$. For a point $p\in A_{b}$, let $\angle(p)$ stands for the inclined angle of the outer radius of $A_{b}$ containing $p$. Then put 
		$$c(p)=\floor{\frac{\angle(p)}{(2/9)\pi}} \bmod 3.$$
		Less formally, the color is assigned according to the angle and it is changed cyclically after an interval of length $(2/9)\pi$ (in terms of angle). The coloring is depicted on Figure~\ref{fig:radial1}. The additionally marked distances  will be explained later in the proof.
		\begin{figure}[h]
			\begin{center}
				\includegraphics[width=0.3\textwidth]{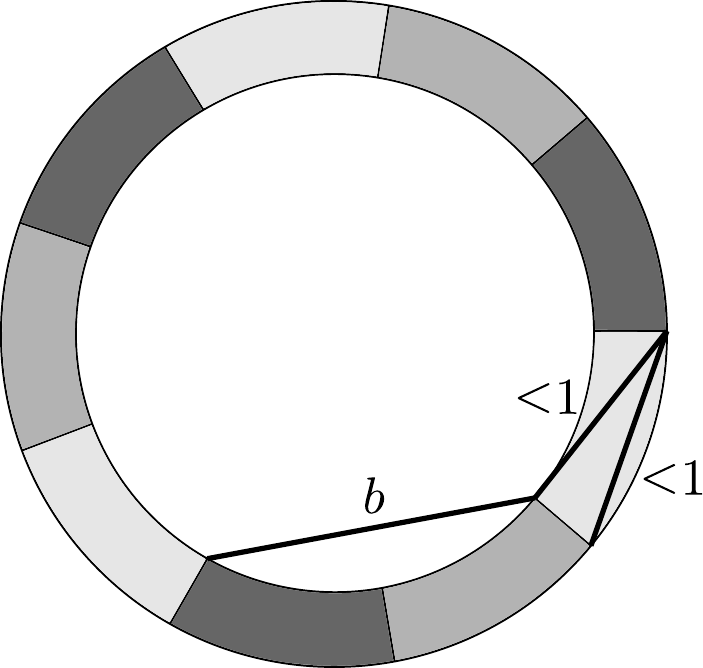}
				\caption{A proper $3$-coloring of $G_{[1,b]}[A_{b}]$ for $b=\sqrt{2 - 2 \sin(\frac{\pi}{18})}$.}
				\label{fig:radial1}
			\end{center}
		\end{figure}
		
		Let us argue that $c$ is a proper coloring of $G_{[1,b]}[A_{b}]$. 
		First, the distance between two points inside a single connected component of a color is (strictly) smaller than the diameter of this component. Note that strictness follows from the choice of colors on boundaries of single-colored regions. On the other hand, this diameter is the maximum of the following two distances: between a pair of points on the outer circle with angle difference $(2/9)\pi$ and between a pair of a point on each the outer and the inner circle with angle difference $(2/9)\pi$. Denote these values by $d_1$ and $d_2$, respectively.
		By the law of cosines we have
		$$d_1 = \sqrt{2b^2-2b^2\cos((2/9)\pi)}  < 1,$$
		$$d_2 = \sqrt{b^2+1^2-2b\cos((2/9)\pi)}  < 1.$$
		
		Secondly, the minimal distance between two points of the same color but from different connected components of this color is (strictly) larger than the distance between two points on the inner circle with angle difference $2\cdot(2/9)\pi$. As before, the strictness of this inequality follows from the choice of colors on boundaries of single-colored regions.
		By the law of cosines, this is equal to
		$$\sqrt{2-2\cos((4/9)\pi)} = \sqrt{2 - 2 \sin(\frac{\pi}{18})} = b.$$

		\item We will show that $\chi(G_{[1,b]}[A_{b}]) \le 4$ for $b=\sqrt{2}$.
		
		For a point $p\in A_{b}$, define a radial proper $4$-coloring $c$ of $G_{[1,b]}[A_{b}]$ by 
		$$c(p)=\floor{\frac{\angle(p)}{(1/6)\pi}} \bmod 4.$$
		
		First, the distance between two points inside a single color connected component is (strictly) smaller than the maximum of the following two distances: between a pair of points on the outer circle with angle difference $(1/6)\pi$ and between a pair of a point on each the outer and the inner circle with angle difference $(2/12)\pi$. Denote these values by $d_1$ and $d_2$, respectively.
		By the law of cosines we have
		$$d_1 = \sqrt{2b^2-2b^2\cos((1/6)\pi)}  < 1,$$
		$$d_2 = \sqrt{b^2+1^2-2b\cos((1/6)\pi)}  < 1.$$
		
		Secondly, the minimal distance between two points of the same color but from different connected components of this color is (strictly) larger than the distance between two points on the inner circle with angle difference $3\cdot(1/6)\pi$.  By the law of cosines, this is equal to
		$$\sqrt{2-2\cos((3/6)\pi)} = b.$$

		\item 
		In order to show that $\chi(G_{[1,b]}[A_{b}]) \le 5$ for $b=\sqrt{\frac{3}{2} - \frac{\sqrt{5}}{2}}$, one can consider the following proper $5$-coloring $c$. For a point $p\in A_{b}$, define $$c(p)=\floor{\frac{\angle(p)}{(1/5)\pi}} \bmod 5.$$
		
		%
		
		\item In order to show that $\chi(G_{[1,b]}[A_{b}]) \le 6$ for $b=\sqrt{3}$, one can consider the following proper $6$-coloring $c$. For a point $p\in A_{b}$, define $$c(p)=\floor{\frac{\angle(p)}{(1/6)\pi}} \bmod 6.$$
		
		\item 
		In order to show that $\chi(G_{[1,b]}[A_{b}]) \le 7$ for $b= 2 \cos(\frac{\pi}{7})$, one can consider the following proper $7$-coloring $c$. For a point $p\in A_{b}$, define $$c(p)=\floor{\frac{\angle(p)}{(1/7)\pi}} \bmod 7.$$
		
		\item 
		In order to show that $\chi(G_{[1,b]}[A_{b}]) \le 8$ for $b=\sqrt{2+\sqrt{2}}$, one can consider the following proper $8$-coloring $c$. For a point $p\in A_{b}$, define $$c(p)=\floor{\frac{\angle(p)}{(1/8)\pi}} \bmod 8.$$
	\end{enumerate}
\end{proof}

As mentioned before, the proof of Theorem~\ref{thm:main} relies on lower bounds for $\chi(G_{[1,b]}[A_{b,\varepsilon}])$ for certain values of $b$ and sufficiently small $\varepsilon>0$. In particular, the proof of Theorem~\ref{thm:main}.1 uses the inequality $\chi(G_{[1,b]}[A_{b,\varepsilon}])\ge 4$ for any $b> \sqrt{2-2 \sin(\frac{18 \pi}{325})} \approx 1.28599$ and sufficiently small $\varepsilon>0$. 
On the other hand, $3$ colors are enough for slightly smaller values of~$b$. Namely, $\chi(G_{[1,b]}[A_{b,\varepsilon}])\le 3$ for any $b\le \sqrt{2 - 2 \sin(\frac{\pi}{18})}\approx 1.28558$ and $\varepsilon>0$ is implied by Proposition~\ref{prop:col_of_annuli} and the inequality $\chi(G_{[1,b]}[A_{b,\varepsilon}]) \le \chi(G_{[1,b]}[A_{b}])$ for any $\varepsilon>0$.
This means that the condition on $b$ in Theorem~\ref{thm:main}.1 cannot be substantially improved with the same method. Unfortunately, for larger values of $b$, this gap is getting much larger. Hence for those values, we do not know if the constructions given in the proof of Theorem~\ref{thm:main} are close to the potential optima. We believe that more complex colorings of annuli need to be constructed to reduce the gap.


\section{Upper bounds on number of colors for the plane} \label{sec:upper}
\subsection{Eight colors}

In this subsection, we present that $8$ colors allows to color $G_{[1,b]}$ for bigger $b$, than for $7$ colors. The coloring was inspired by a $7$-coloring of $G_{\{1\}}$ by Edward Pegg, Jr. \cite{Soifer2009}, such that the seventh color occupies only about $1/3$ of $1$\% of the plane.

\begin{theorem}\label{thm:8-col}
For $b=1.37542$ holds $\chi(G_{[1, b]})\le 8$
\end{theorem}
\begin{proof}
The coloring is given by Figure \ref{trojkaciki} and it is a modified classic 7-coloring (see Figure~\ref{fig:7col-eps}) by adding a small triangle in the eighth color in every second meeting point of their hexagons. This allows to make hexagons bigger and in consequence enlarge $b$. Let $x,y$ be defined by the Figure \ref{trojkacikiZblizenie}.

\begin{figure}[h]\begin{center}
\includegraphics[width=0.8\textwidth]{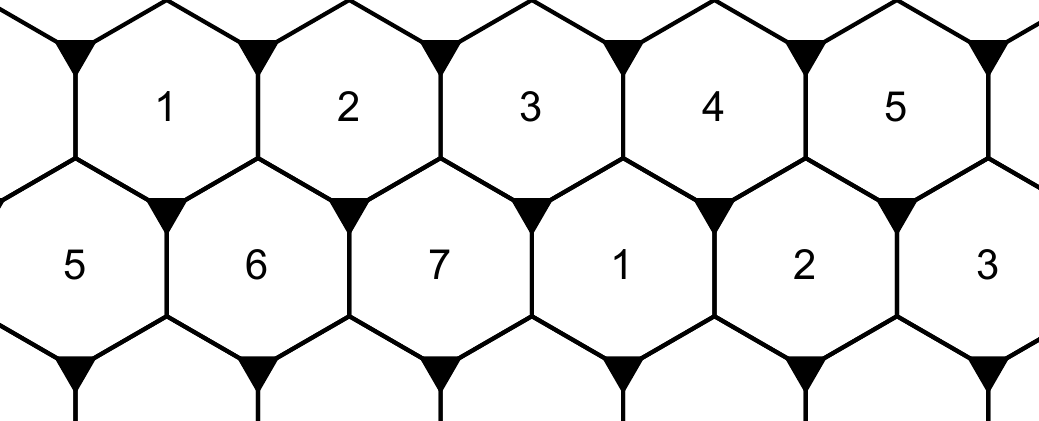}
\caption{8-coloring, black triangles are in color number 8}
\label{trojkaciki}
\end{center}
\end{figure}

\begin{figure}[h]\begin{center}
\includegraphics[width=0.8\textwidth]{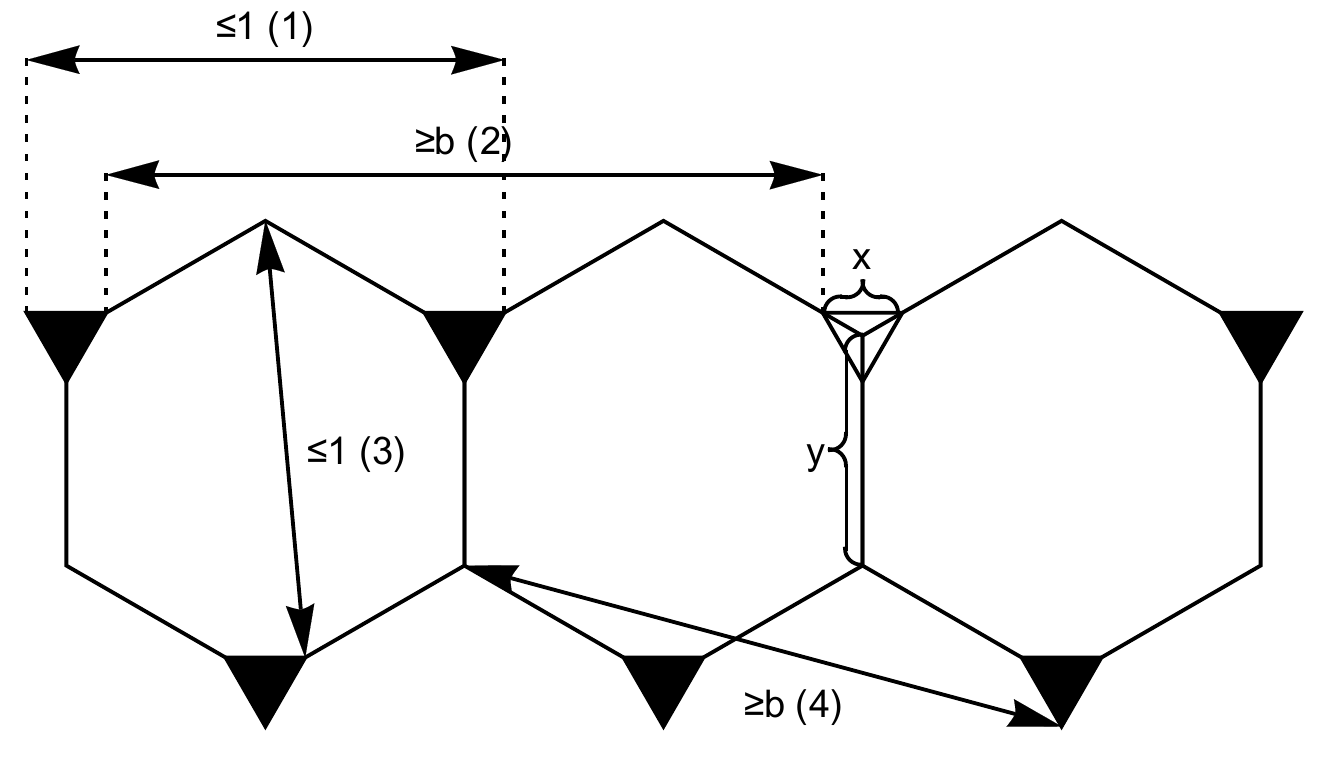}
\caption{8-coloring - restrictions}
\label{trojkacikiZblizenie}
\end{center}
\end{figure}

 To make the coloring proper, $x$ and $y$ must fulfill restrictions: 
\begin{enumerate}
\item $y\sqrt{3}+x\le 1$
\item $2y\sqrt{3}-x\ge b$
\item $ \left(2 y - \frac{x}{2 \sqrt{3}}\right)^2 + \left(\frac{x}{2}\right)^2\le 1$
\item $\left(\frac{3}{2}y\sqrt{3}\right)^2+\left(\frac{y}{2}+\frac{x}{\sqrt{3}}\right)^2\ge b^2$
\end{enumerate}

Under these restrictions, we maximize $b$ obtaining $b=1.37542$ for $y=0.514884$ and $x=0.108194$. 
\end{proof}

\subsection{General method for coloring $G_{[1,b]}$}

A few general bounds on the number of colors for larger values of parameter $b$ are given in the following theorems:

\begin{theorem}[Exoo \cite{Exoo2005}] \label{exoo} 
If $b<\frac{1}{2}\sqrt{9r^2-3r+1}$, then $\chi(G_{[1,b]})\le 3r^2+3r+1$.
\end{theorem}

Actually, Exoo stated his result for the graph $G_{[1-\varepsilon,1+\varepsilon]}$, but as we mentioned before his notion can be translated easily into ours.

Lonc \cite{lonc} generalized idea of $12$-coloring by Ivanow \cite{Ivanov2007} and obtained improvement  for some values of $b$:

\begin{theorem}[Lonc \cite{lonc}] \label{lonc} 
If $b\leq\frac{3}{2}r-1$, then $\chi(G_{[1,b]})\le 3r^2$.
\end{theorem}


The next bound (again partially improving previous ones) follows as a special case of a fractional multifold coloring scheme established in \cite{GrytczukJunoszaSokolWesek2016}. 

\begin{theorem}[Chybowska-Sokół, Grytczuk, Junosza-Szaniawski, Węsek \cite{GrytczukJunoszaSokolWesek2016}]
If $b\leq\frac{\sqrt{3}}{2}(r-1)$, then $\chi(G_{[1,b]})\le r^2$.\label{1warstwowe}
\end{theorem}

The aim of this subsection is to present a general method for coloring $G_{[1,b]}$ that improves the previously known results in some ranges of $b$.
The Figure \ref{wykres1} shows the comparison of bounds from \cite{Exoo2005, lonc, GrytczukJunoszaSokolWesek2016} and this paper on the number of colors depending on the value of $b$.  

We define a scheme for constructing a coloring of the plane based on hexagonal tilings. All points in a single hexagon (including some borders as on figure \ref{fig:7col_1}) will be colored with the same color. 
 Let us start with defining the tiling. Let $H_{0,0}$ be a hexagon with two vertical sides, center in $(0,0)$, diameter equal one, and part of the boundary removed as in Figure \ref{heksagon}. Note that the width of $H_{0,0}$ equals $\frac{\sqrt{3}}{2}$.
Then let $s_1=[\frac{\sqrt{3}}{2},0]$, and $s_2=[\frac{\sqrt{3}}{4},-\frac{3}{4}]$. For $i,j\in \mathbb{Z}$ let $H_{i,j}$, be a tile created by  shifting $H_{0,0}$ by a vector $i\cdot s_1+j\cdot s_2$, namely $H_{i,j}:=\{(x,y)+i\cdot s_1+j\cdot s_2\in\RR:\ (x,y)\in H_{0,0}\}$.  
Notice that $\{H_{i,j}:i,j\in \mathbb{Z}\}$ forms a partition of the plane, which we call a hexagonal tiling.

Now for a pair of non-negative integers $(p,q)$, we define a coloring such that for any $(i,j)\in \mathbb{Z}^2$ hexagons $H_{i,j}, H_{i+p,j+q},H_{i+p+q,j-p}$ have the same color. Notice that the centers of such three hexagons form an equilateral triangle (see Figure \ref{7pattern}). By reapplying this rule of a single colored triangles, we obtain that sets of the form $\{H_{i+k\cdot p + l\cdot (p+q),j+k\cdot q-l\cdot p}:\ k,l\in \mathbb{Z}\}$ are monochromatic.  
For any $q, p$, the \emph{$(p,q)$-coloring} is the coloring in which all color classes are of this form.


Let us define the distance between two hexagons as the infimum of the distances between points from the hexagons, i.e. $$\dist (H_{i1,j1},H_{i2,j2})=\inf\{\dist(p_1,p_2):\ p_1\in H_{i1,j1}\ \textnormal{and}\ p_2\in H_{i2,j2}\}.$$

\begin{figure}[h] \centering
\subfigure[A boundary of a tile.]{\label{heksagon}
\includegraphics[height=2cm]{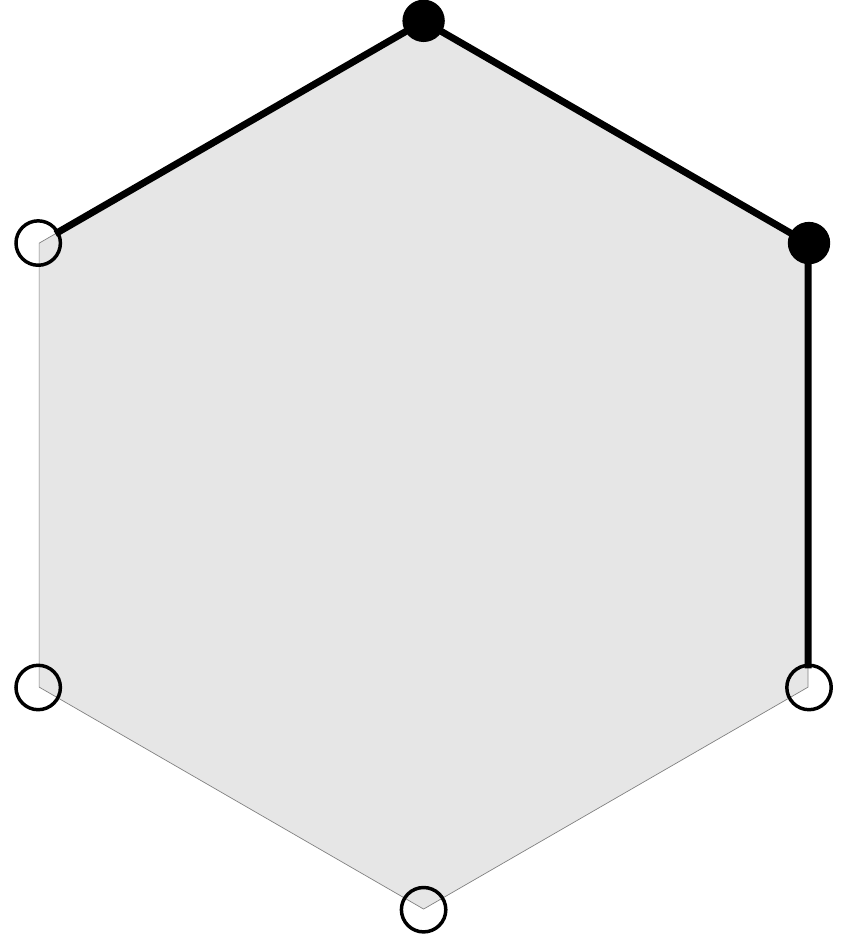}
}
\subfigure[Relation between $H_{0,0}$ and $H_{i,j}$.]{\label{Hij}\includegraphics[width=0.4\textwidth]{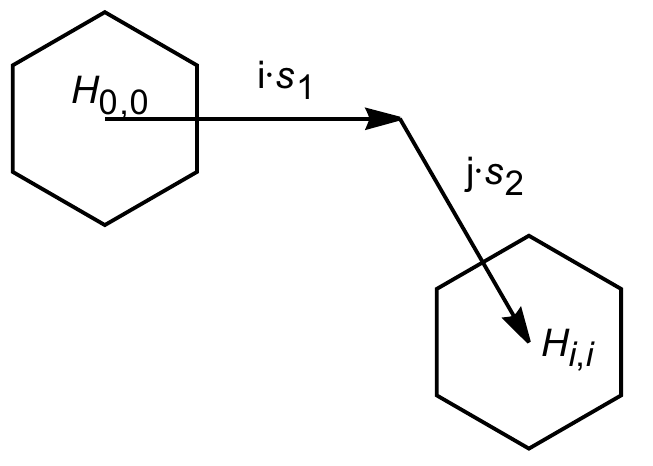}}
\subfigure[Hexagonal tiling.]{\label{tiling}\includegraphics[width=0.4\textwidth]{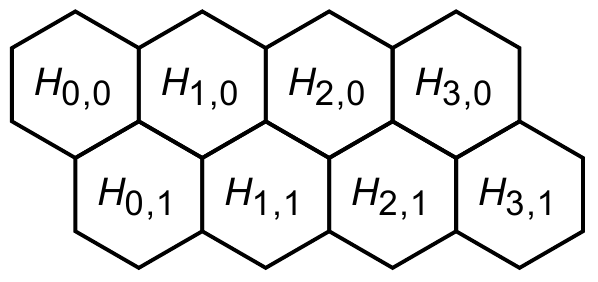}}
\caption{Hexagonal tiling.}
\end{figure}

\begin{figure}[h]
\center\includegraphics[scale=.7]{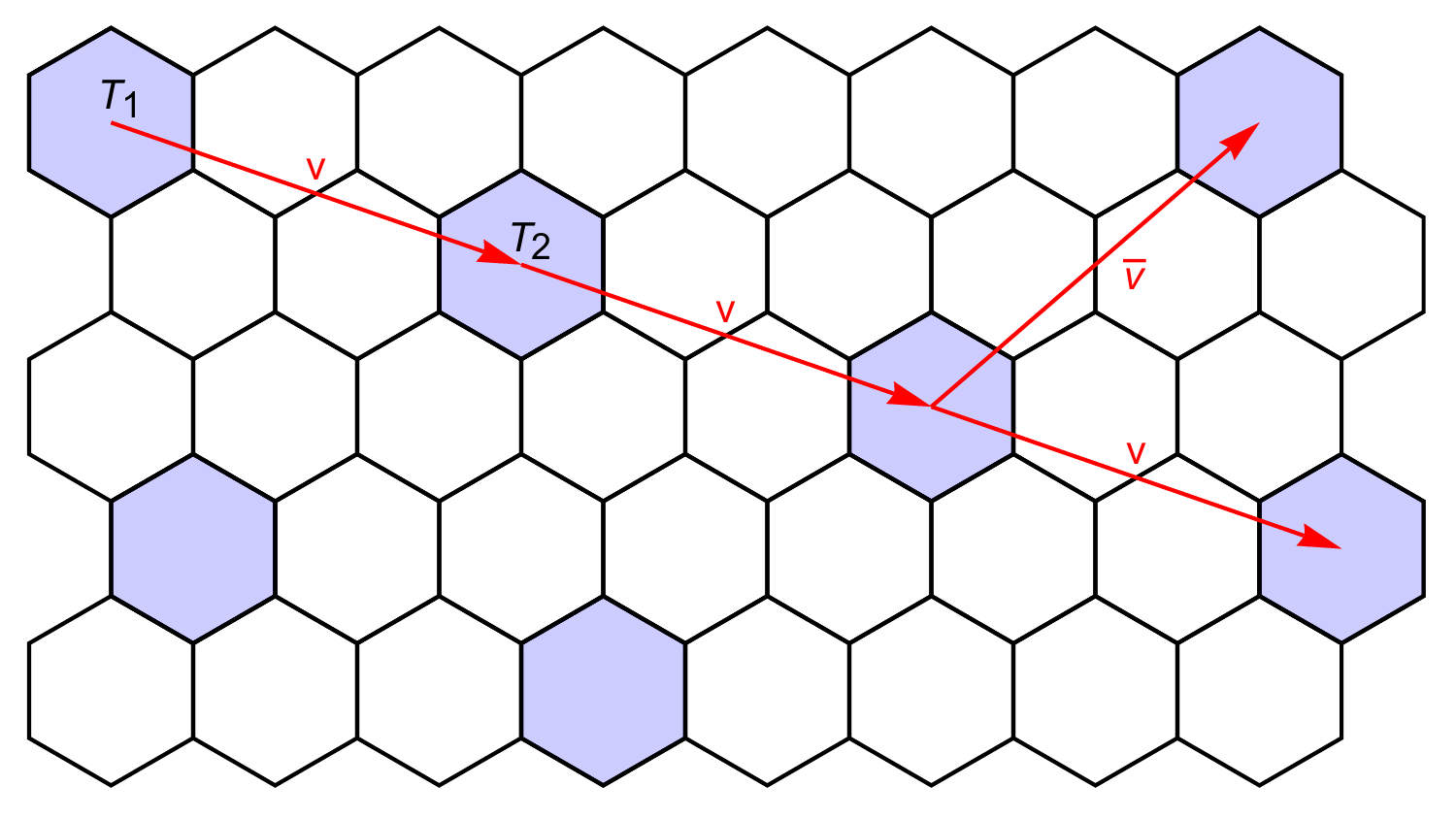}\caption{A pattern given by tiles $H_{0,0}$ and $H_{p,q}$}\label{7pattern}
\end{figure}


\begin{lemma}
If $p,q\ge 1$ then  $(p,q)$-coloring uses $p^2+q^2+p q$ colors.
\end{lemma}
Let us denote greatest common divisor of $p$ and $q$ as $d$ and let $p'=p/d$, $q'=q/d$.

We call the set of $H_{i,j}$, $i\in\mathbb{Z}$, the \emph{$j$-th row} of tiles. Let us call the color of $H_{0,0}$ blue and $\mathcal{T}=\{H_{k\cdot p + l\cdot (p+q),k\cdot q-l\cdot p}:\ k,l\in \mathbb{Z}\}$ denote the set of all blue hexagons. To give some insight of where $\mathcal{T}$ comes from, let $v$ denote a vector from $(0,0)$ to the center of $H_{p,q}$ ($v=p\cdot s_1+q\cdot s_2$) and $\overline{v}$ be the vector obtained from $v$ via rotating it by $\frac{\pi}{3}$ ($\overline{v}=p\cdot (s_1 - s_2)+ q\cdot s_1=(p+q)\cdot s_1-p\cdot s_2$). Then $\mathcal{T}$ is a set of all tiles created by shifting $H_{0,0}$ by $k v + l \overline{v}$, for $k,l\in \mathbb{Z}$.
First let us notice that, by the definition of $\mathcal{T}$, blue appears in rows numbered by $k\cdot q - l\cdot p=(kq'- lp')\cdot d$, for any $k,l\in \mathbb{Z}$. 
Since $p'$ and $q'$ are coprime, $\{kq'- lp':\ k,l\in \mathbb{Z}\}=\mathbb{Z}$.
 Then blue appears in a row if and only if its number is divisible by $d$.

Note that, since we used the same pattern (only shifted) for every color, we have the same number of colors in every row. Hence the total number of colors equals the number of colors used in a single row multiplied by $d$.

In order to find the number of colors used in a row, it is enough to know how often does a single color reappear in it (see example on Figure \ref{fig:7col-eps}) in one row every 7$^{\text{th}}$ hexagon is blue and we use 7 colors in each row). 


Let $m$ be the smallest positive number such that $H_{m,0}$ is colored blue. Since $H_{m,0}\in\mathcal{T}$, then there exist integers $k,l$ such that:\begin{align}
k\cdot p + l\cdot (p+q)=m,\label{e1}\\
k\cdot q-l\cdot p=0.\label{e2}
\end{align}

From (\ref{e2}) we derive $kq'=lp'$. Since $p',\ q'$ are coprime, then $k$ is divisible by $p'$ and $l$ by $q'$. Moreover $p',\ q'$ are both non-negative, hence either $k$ and $l$ are both non-positive or they are both non-negative. Without loss of generality we assume the latter.
Hence $m=d\cdot(k p' + l q' +l p')$ is minimal when $k=p'$ and $l=q'$. So the numbers of colors in a row equals $m=d(p'p' + q'q' +p'q')$, and from our previous remarks we conclude that the total number of colors equals $d^2(p'p' + q'q' +p'q')=p^2+q^2+pq$.


\begin{remark} The coloring from Exoo Theorem \ref{exoo} is $(r,r+1)$-coloring, coloring from Lonc Theorem \ref{lonc}  is $(r,r)$-coloring and coloring from \ref{1warstwowe} is $(r,0)$-coloring.  
\end{remark}
\begin{figure}[h]\begin{center}
\includegraphics[width=0.8\textwidth]{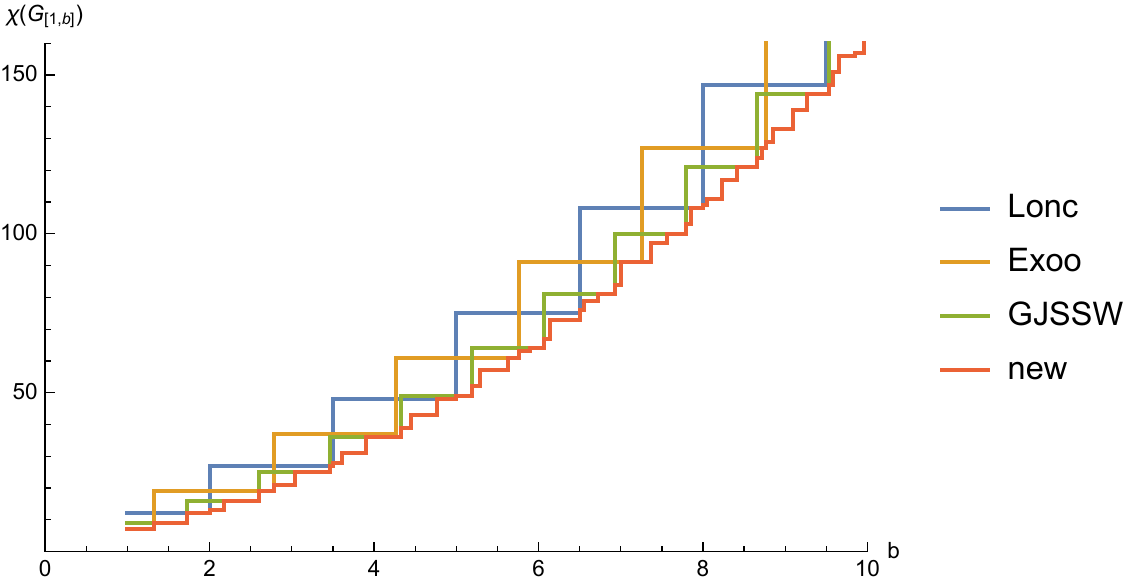}
\caption{The plot shows bounds on $\chi(G_{[1,b]})$ given by the theorems above and our best results. 
}

\label{wykres1}
\end{center}\end{figure}

Table~\ref{tab:bounds} presents bounds on $\chi(G_{[1,b]}[\mathbb{R}^2])$.
\begin{center}
	\begin{table}[h]
		\begin{tabular}{ | c|  c|  c | c | c | c|}
			\hline
			$b$	&	$b\approx $	&	$\chi(G{[1,b]})\le $&		p	&	q & First appears in	\\
			\hline
$	 \sqrt{7}/2	$ &	1,32288	&	7	&	1	&	2	&	Isbell;  Hadwiger \cite{Hadwiger1961}	\\
$	 \sqrt{3}	$ &	1,73205	&	9	&	0	&	3	&	Ivanov \cite{Ivanov2007}	\\
$	2	$ &	2	&	12	&	2	&	2	&	Ivanov \cite{Ivanov2007}	\\
$	 \sqrt{19}/2	$ &	2,17945	&	13	&	1	&	3	&	Ivanov \cite{Ivanov2007}	\\
$	 (3 \sqrt{3})/2	$ &	2,59808	&	16	&	0	&	4	&	Ivanov \cite{Ivanov2007}	\\
$	 \sqrt{31}/2	$ &	2,78388	&	19	&	2	&	3	&	Exoo \cite{Exoo2005}	\\
$	\sqrt{37}/2	$ &	3,04138	&	21	&	1	&	4	&	Ivanov \cite{Ivanov2007}	\\
$	 2 \sqrt{3}	$ &	3,4641	&	25	&	0	&	5	&	GJSSW \cite{GrytczukJunoszaSokolWesek2016}	\\
$	 7/2	$ &	3,5	&	27	&	3	&	3	&	Lonc \cite{lonc}	\\
$	 \sqrt{13}	$ &	3,60555	&	28	&	2	&	4	&	new	\\
$	 \sqrt{61}/2	$ &	3,90512	&	31	&	1	&	5	&	new	\\
$	 (5 \sqrt{3})/2	$ &	4,33013	&	36	&	0	&	6	&	GJSSW \cite{GrytczukJunoszaSokolWesek2016}	\\
$	 \sqrt{79}/2	$ &	4,4441	&	39	&	2	&	5	&	new	\\
$	\sqrt{91}/2	$ &	4,7697	&	43	&	1	&	6	&	new	\\
$	5	$ &	5	&	48	&	4	&	4	&	Lonc \cite{lonc}	\\
$	 3 \sqrt{3}	$ &	5,19615	&	49	&	0	&	7	&	GJSSW \cite{GrytczukJunoszaSokolWesek2016}	\\
$	 2 \sqrt{7}	$ &	5,2915	&	52	&	2	&	6	&	new	\\
$	 \sqrt{127}/2	$ &	5,63471	&	57	&	1	&	7	&	new	\\
$	\sqrt{133}/2	$ &	5,76628	&	61	&	4	&	5	&	Exoo \cite{Exoo2005}	\\
$	 \sqrt{139}/2	$ &	5,89491	&	63	&	3	&	6	&	new	\\
$	 (7 \sqrt{3})/2	$ &	6,06218	&	64	&	0	&	8	&	GJSSW \cite{GrytczukJunoszaSokolWesek2016}	\\
$	 \sqrt{151}/2	$ &	6,1441	&	67	&	2	&	7	&	new	\\
$	13/2	$ &	6,5	&	75	&	5	&	5	&	Lonc \cite{lonc}	\\
$	 \sqrt{43}	$ &	6,55744	&	76	&	4	&	6	&	new	\\
$	 \sqrt{181}/2	$ &	6,72681	&	79	&	3	&	7	&	new	\\
$	 4 \sqrt{3}	$ &	6,9282	&	81	&	0	&	9	&	GJSSW \cite{GrytczukJunoszaSokolWesek2016}	\\
$	7	$ &	7	&	84	&	2	&	8	&	new	\\
$	 \sqrt{211}/2	$ &	7,26292	&	91	&	5	&	6	&	Exoo \cite{Exoo2005}	\\
$	 \sqrt{217}/2	$ &	7,36546	&	93	&	4	&	7	&	new	\\
$	 \sqrt{229}/2	$ &	7,56637	&	97	&	3	&	8	&	new	\\
$	 (9 \sqrt{3})/2	$ &	7,79423	&	100	&	0	&	10	&	GJSSW \cite{GrytczukJunoszaSokolWesek2016}	\\
$	 \sqrt{247}/2	$ &	7,85812	&	103	&	2	&	9	&	new	\\
$	8	$ &	8	&	108	&	6	&	6	&	Lonc \cite{lonc}	\\
$	 \sqrt{259}/2	$ &	8,04674	&	109	&	5	&	7	&	new	\\
$	\sqrt{271}/2	$ &	8,23104	&	111	&	1	&	10	&	new	\\
$	 \sqrt{283}/2	$ &	8,4113	&	117	&	3	&	9	&	new	\\
$	2 \sqrt{19}	$ &	8,7178	&	124	&	2	&	10	&	new	\\
$	 \sqrt{307}/2	$ &	8,76071	&	127	&	6	&	7	&	Exoo \cite{Exoo2005}	\\
$	\sqrt{313}/2	$ &	8,8459	&	129	&	5	&	8	&	new	\\
$	 (5 \sqrt{13})/2	$ &	9,01388	&	133	&	4	&	9	&	new	\\
$	 (7 \sqrt{7})/2	$ &	9,26013	&	139	&	3	&	10	&	new	\\
$	 19/2	$ &	9,5	&	147	&	7	&	7	&	Lonc \cite{lonc}	\\
$	\sqrt{91}	$ &	9,53939	&	148	&	6	&	8	&	new	\\
$	 \sqrt{373}/2	$ &	9,6566	&	151	&	5	&	9	&	new	\\
$	\sqrt{97}	$ &	9,84886	&	156	&	4	&	10	&	new	\\
$	 \sqrt{421}/2	$ &	10,2591	&	169	&	7	&	8	&	Exoo \cite{Exoo2005}	\\
$	\sqrt{427}/2	$ &	10,332	&	171	&	6	&	9	&	new	\\
$	 \sqrt{439}/2	$ &	10,4762	&	175	&	5	&	10	&	new	\\
$	11	$ &	11	&	192	&	8	&	8	&	Lonc \cite{lonc}	\\
$	 \sqrt{487}/2	$ &	11,034	&	193	&	7	&	9	&	new	\\
$	 2 \sqrt{31}	$ &	11,1355	&	196	&	6	&	10	&	new	\\
$	 \sqrt{553}/2	$ &	11,758	&	217	&	8	&	9	&	Exoo \cite{Exoo2005}	\\
$	\sqrt{559}/2	$ &	11,8216	&	219	&	7	&	10	&	new	\\
$	 25/2	$ &	12,5	&	243	&	9	&	9	&	Lonc \cite{lonc}	\\
$	\sqrt{157}	$ &	12,53	&	244	&	8	&	10	&	new	\\
$	 \sqrt{703}/2	$ &	13,2571	&	271	&	9	&	10	&	Exoo \cite{Exoo2005}	\\
$	14	$ &	14	&	300	&	10	&	10	&	Lonc \cite{lonc}	\\
			\hline
			
		\end{tabular}
		\caption{Bounds on $\chi(G_{[1,b]}[A_{b}])$ depending on $b$ achieved by $(p,q)$-coloring. }
		\label{tab:bounds}
	\end{table}
\end{center}

\section{Conclusions}

We note that our method of constructing lower bounds combines theoretical reasoning of continuous nature and constructions of finite sets for which the coloring properties are checked by computer. Therefore, the approach differs from the previously used in the literature. Similar constructions for larger values of $b$ and a larger number of colors can be designed.
However, it seems that the method should work better for relatively small values of $b$ (and hence a small number of colors), as in this case the $3$ colors reserved by the $\varepsilon$-ball make a greater difference.

One may observe that we do not have any interval with the chromatic number determined to $8$, $10$ or $11$ colors - even though Theorem~\ref{thm:main} gives some results concerning these numbers of colors. The reason seems to be that we do not have sufficiently good $k$-colorings of the plane for $k\in \{8,10,11\}$. That is, a~$k$-coloring of $G_{[1,b]}$ that would work for sufficiently large $b$. In particular, let us consider $8$ colors. In Theorem~\ref{thm:8-col}, we present an $8$-coloring for $b=1.37542$. However, it is still not enough, as the corresponding lower bound from Theorem~\ref{thm:main} works only for $b>\sqrt{2+2\sin(\frac{\pi}{38})} \approx 1.47145$. The gap is relatively large and seems that closing it, if possible, would require a new idea -- note that Theorem~\ref{thm:8-col} is based on careful use of the coloring scheme of the classic $7$-coloring (see Figure~\ref{fig:7col-eps}). It would be interesting to obtain an $8$-coloring of $G_{[1,b]}$ even for $b>1.4$. Moreover, we do not have any $10$-coloring or any $11$-coloring for $b>\sqrt{3}$, while $b=\sqrt{3}$ is satisfied by the $9$-coloring from Figure~\ref{fig:9col}. May it be that adding $2$ colors to these $9$ colors does not make any difference? Let us conclude this part of the discussion with the following open problem:
\begin{conjecture}
	For any integer $k\ge 7$, there exists $b>1$ such that $\chi(G_{[1,b]})=k$.
\end{conjecture}


Our 8-coloring does not have the property which in \cite{GrytczukJunoszaSokolWesek2016} is called solid coloring, that is every color class consists of regions pairwise at a distance of at least one. Solid colorings can be applied in online colorings of disc intersections graphs. We post as an open question if there is a solid 8-coloring of $G_{[1,b]}$ for any $b>\sqrt{7}/2$. 

For general upper bounds, we concentrate on the colorings of the Euclidean plane based on hexagonal tiling.
However, this method can be adjusted to the setting of other regular tilings of the plane. It can also be applied to coloring Euclidean spaces of higher dimensions.

\section*{Acknowledgement}
We solve instances of the problem using IBM ILOG CPLEX solver (version 12.7.1) and ZIMPL model generating language \cite{Koch2004}. Data will be made available on reasonable request


\begin{thebibliography}{10}

\bibitem{AlmManske2014}
Jeremy~F. Alm and Jacob Manske.
\newblock On radial colorings of annuli.
\newblock {\em Australasian J. Combinatorics}, 60:270--278, 2014.

\bibitem{ArdalManuchRosenfeldShelahStacho2009}
Hayri Ardal, J{\'a}n Ma{\v{n}}uch, Moshe Rosenfeld, Saharon Shelah, and
  Ladislav Stacho.
\newblock The odd-distance plane graph.
\newblock {\em Discrete {\&} Computational Geometry}, 42(2):132--141, Sep 2009.

\bibitem{AxenovichChoiLastrinaMcKaySmithStanton2014}
M.~Axenovich, J.~Choi, M.~Lastrina, T.~McKay, J.~Smith, and B.~Stanton.
\newblock On the chromatic number of subsets of the euclidean plane.
\newblock {\em Graphs and Combinatorics}, 30(1):71--81, 2014.

\bibitem{Bauslaugh1998}
Bruce~L. Bauslaugh.
\newblock Tearing a strip off the plane.
\newblock {\em Journal of Graph Theory}, 29(1):17--33, 1998.

\bibitem{Bock2019}
Felix Bock.
\newblock Epsilon-colorings of strips.
\newblock {\em Acta Mathematica Universitatis Comenianae}, 88(3):469--473,
  2019.

\bibitem{Bourgain1986}
J.~Bourgain.
\newblock A szemer{\'e}di type theorem for sets of positive density inrk.
\newblock {\em Israel Journal of Mathematics}, 54(3):307--316, 1986.

\bibitem{Bukh2008}
Boris Bukh.
\newblock Measurable sets with excluded distances.
\newblock {\em Geometric and Functional Analysis}, 18(3):668--697, 2008.

\bibitem{CherkashinKulikovRaigorodskii2018}
Danila Cherkashin, Anatoly Kulikov, and Andrei Raigorodskii.
\newblock On the chromatic numbers of small-dimensional euclidean spaces.
\newblock {\em Discrete Applied Mathematics}, 243:125--131, 2018.

\bibitem{CranstonRabern2017}
Daniel~W Cranston and Landon Rabern.
\newblock The fractional chromatic number of the plane.
\newblock {\em Combinatorica}, 37(5):837--861, 2017.

\bibitem{CurrieEggleton}
James~D. Currie and Roger~B. Eggleton.
\newblock Chromatic properties of the euclidean plane.
\newblock Preprint, \href{https://arxiv.org/abs/1509.03667} {arXiv:1509.03667}.

\bibitem{CURRIE1992}
J.D. Currie.
\newblock Connectivity of distance graphs.
\newblock {\em Discrete Mathematics}, (1):91 -- 94, 1992.

\bibitem{deGrey2018}
Aubrey de~Grey.
\newblock {The chromatic number of the plane is at least $5$}.
\newblock {\em Geombinatorics}, 28:18--31, 2018.

\bibitem{BruijnErdos1951}
de~Ng~Dick~Bruijn and P.~Erd{\"o}s.
\newblock A colour problem for infinite graphs and a problem in the theory of
  relations.
\newblock volume~13, page 369–373, 1951.

\bibitem{DeVosEbrahimiGheblehGoddynMoharNaserasr2007}
Matt DeVos, J.~Ebrahimi, M.~Ghebleh, L.~Goddyn, B.~Mohar, and R.~Naserasr.
\newblock Circular coloring the plane.
\newblock {\em SIAM J. Discret. Math.}, 21:461--465, 2007.

\bibitem{PalvolgyiAgoston2019}
{Dömötör Pálvölgyi, }.
\newblock Comments in polymath16, thread 14, comments 24460 and 24487.
\newblock
  \url{https://dustingmixon.wordpress.com/2019/08/05/polymath16-fourteenth-thread-automated-graph-minimization/},
  2019.

\bibitem{DunfieldBrownPerryI}
N.~Dunfield, N.~Brown, and G.~Perry.
\newblock Colorings of the plane i.
\newblock {\em Geombinatorics}, III:24--31, 1993.

\bibitem{DunfieldBrownPerryII}
N.~Dunfield, N.~Brown, and G.~Perry.
\newblock Colorings of the plane ii.
\newblock {\em Geombinatorics}, III:64--74, 1994.

\bibitem{DunfieldBrownPerryIII}
N.~Dunfield, N.~Brown, and G.~Perry.
\newblock Colorings of the plane iii.
\newblock {\em Geombinatorics}, III:110--114, 1994.

\bibitem{EggletonErdosSkilton1990}
R.~B. Eggleton, P.~Erd{\"o}s, and D.~K. Skilton.
\newblock Colouring prime distance graphs.
\newblock {\em Graphs and Combinatorics}, 6(1):17--32, 1990.

\bibitem{EggletonErdosSkilton1985}
R.~B. Eggleton, P.~Erdös, and D.~K. Skilton.
\newblock Colouring the real line.
\newblock {\em Journal of Combinatorial Theory, Series B}, 39(1):86 -- 100,
  1985.

\bibitem{Exoo2005}
G.~Exoo.
\newblock $\varepsilon$-unit distance graphs.
\newblock {\em Discrete \& Computational Geometry}, 33:117--123, 2005.

\bibitem{ExooIsmailescu2014}
Geoffrey Exoo and Dan Ismailescu.
\newblock On the chromatic number of $\mathbb{R}^n$ for small values of $n$,
  2014.

\bibitem{ExooIsmailescu2017}
Geoffrey Exoo and Dan Ismailescu.
\newblock A unit distance graph in the plane with fractional chromatic number
  383/102.
\newblock {\em Geombinatorics}, 26(3):122--127, 2017.

\bibitem{ExooIsmailescu2018}
Geoffrey Exoo and Dan Ismailescu.
\newblock {The Hadwiger-Nelson problem with two forbidden distances}.
\newblock {\em Geombinatorics}, 28:51--68, 2018.

\bibitem{ExooIsmailescu2019}
Geoffrey Exoo and Dan Ismailescu.
\newblock The chromatic number of the plane is at least 5: A new proof.
\newblock {\em Discrete \& Computational Geometry}, pages 1--11, 2019.

\bibitem{ExooIsmailescu2020}
Geoffrey Exoo and Dan Ismailescu.
\newblock {A $6$-chromatic two-distance graph in the plane}.
\newblock {\em Geombinatorics}, 29(3), 2020.

\bibitem{ExooIsmailescuLim2014}
Geoffrey Exoo, Dan Ismailescu, and Michael Lim.
\newblock On the chromatic number of $\mathbb{R}^4$.
\newblock {\em Discrete \& Computational Geometry}, 52(2):416--423, 2014.

\bibitem{FalconerMarstrand1986}
K.~J. Falconer and J.~M. Marstrand.
\newblock {Plane Sets with Positive Density at Infinity Contain all Large
  Distances}.
\newblock {\em Bulletin of the London Mathematical Society}, 18(5):471--474,
  1986.

\bibitem{Falconer1981}
K.J. Falconer.
\newblock The realization of distances in measurable subsets covering
  $\mathbb{R}^n$.
\newblock {\em Journal of Combinatorial Theory, Series A}, 31(2):184--189,
  1981.

\bibitem{FurstenbergKatznelsonWeiss1990}
Hillel F{\"u}rstenberg, Yitzchak Katznelson, and Benjamin Weiss.
\newblock {\em Ergodic Theory and Configurations in Sets of Positive Density},
  pages 184--198.
\newblock Springer Berlin Heidelberg, Berlin, Heidelberg, 1990.

\bibitem{GrytczukJunoszaSokolWesek2016}
Jarosław Grytczuk, Konstanty Junosza{-}Szaniawski, Joanna Sokół, and
  Krzysztof Węsek.
\newblock Fractional and j-fold coloring of the plane.
\newblock {\em Discrete {\&} Computational Geometry}, 55:594--609, 2016.

\bibitem{Hadwiger1945}
Hugo Hadwiger.
\newblock Überdeckung des euklidischen raum durch kongruente mengen.
\newblock {\em Portugaliae Math.}, 4:238–--242, 1945.

\bibitem{Hadwiger1961}
Hugo Hadwiger.
\newblock Ungeloste probleme.
\newblock {\em Elemente der Mathematik}, 16:103--104, 1961.

\bibitem{Heule2018_5}
Marijn J.~H. Heule.
\newblock Computing small unit-distance graphs with chromatic number 5.
\newblock {\em Geombinatorics}, 28:32–--50, 2018.

\bibitem{Heule2018_6}
Marijn J.~H. Heule.
\newblock Searching for a unit-distance graph with chromatic number 6,.
\newblock In Marijn J.~H. Heule, Matti Järvisalo, and Martin Suda, editors,
  {\em SAT COMPETITION 2018}, page~66. University of Helsinki, 2018.

\bibitem{Heule2019}
Marijn J.~H. Heule.
\newblock Trimming graphs using clausal proof optimization.
\newblock In Thomas Schiex and Simon de~Givry, editors, {\em Principles and
  Practice of Constraint Programming}, pages 251--267. Springer International
  Publishing, 2019.

\bibitem{Ivanov2007}
L.~L. Ivanov.
\newblock On the chromatic numbers of $\mathbb{R}^2$ and $\mathbb{R}^3$ with
  intervals of forbiden distances.
\newblock {\em Electronic Notes in Discrete Mathematics}, 29:159--162, 2007.

\bibitem{lonc}
Z.~Lonc J.~Grytczuk, K. Junosza-Szaniawski.
\newblock Colouring (a,b)-distance graphs, 2007.
\newblock talk on Colourings, Independence and Domination - workshop on graph
  theory, Karpacz.

\bibitem{Parts2019_FCN_general}
{J. Parts}.
\newblock Comments in polymath16, thread 12, comment 23601.
\newblock
  \url{https://dustingmixon.wordpress.com/2019/03/23/polymath16-twelfth-thread-year-in-review-and-future-plans/},
  2019.

\bibitem{Parts2019_FCN_small}
{J. Parts}.
\newblock Comments in polymath16, thread 14, comment 22583.
\newblock
  \url{https://dustingmixon.wordpress.com/2019/03/23/polymath16-twelfth-thread-year-in-review-and-future-plans/},
  2019.

\bibitem{JunoszaSzaniawski2018}
Konstanty Junosza-Szaniawski.
\newblock Upper bound on the circular chromatic number of the plane.
\newblock {\em The Electronic Journal of Combinatorics}, 25(1):1--53, 2018.

\bibitem{KatzKrebsShaheen2014}
Richard Katz, Mike Krebs, and Anthony Shaheen.
\newblock Zero sums on unit square vertex sets and plane colorings.
\newblock {\em The American Mathematical Monthly}, 121(7):610--618, 2014.

\bibitem{Katznelson2001}
Y.~Katznelson.
\newblock {Chromatic Numbers of Cayley Graphs on $\mathbb{Z}$ and Recurrence}.
\newblock {\em Combinatorica}, 21(2):211--219, 2001.

\bibitem{Kloeckner2015}
B.~R. Kloeckner.
\newblock Coloring distance graphs: A few answers and many questions.
\newblock {\em Geombinatorics}, 24:117–--134, 2015.

\bibitem{Koch2004}
Thorsten Koch.
\newblock {\em Rapid Mathematical Prototyping}.
\newblock PhD thesis, Technische Universit{\"a}t Berlin, 2004.

\bibitem{Krebs2021}
Mike Krebs.
\newblock Finite \$\\epsilon\$-unit distance graphs.
\newblock {\em Journal of Algebra Combinatorics Discrete Structures and
  Applications}, 8:161 -- 166, 2021.

\bibitem{Kruskal2008}
Clyde~P. Kruskal.
\newblock The chromatic number of the plane: The bounded case.
\newblock {\em Journal of Computer and System Sciences}, 74(4):598 -- 627,
  2008.

\bibitem{MoserMoser1961}
M.~J. Nielsen.
\newblock Solution to problem 10.
\newblock {\em Canad. Math. Bull.}, 4:187--–189, 1961.

\bibitem{OostemaMartinsHeule2020}
Peter Oostema, Ruben Martins, and Marijn Heule.
\newblock Coloring unit-distance strips using sat.
\newblock In Elvira Albert and Laura Kovacs, editors, {\em LPAR23. LPAR-23:
  23rd International Conference on Logic for Programming, Artificial
  Intelligence and Reasoning}, volume~73 of {\em EPiC Series in Computing},
  pages 373--389. EasyChair, 2020.

\bibitem{OwingsTetivaHuddleston2008}
J.~Owings, M.~Tetiva, and M.~Huddleston.
\newblock Coloring the plane.
\newblock {\em The American Mathematical Monthly}, 115(2):170 -- 172, 2008.

\bibitem{Parts2020_two_distances}
J.~Parts.
\newblock {A small $6$-chromatic two-distance graph in the plane}.
\newblock {\em Geombinatorics}, 29(3):111--115, 2020.

\bibitem{Parts2020_minimization}
J.~Parts.
\newblock {Graph minimization, focusing on the example of 5-chromatic
  unit-distance graphs in the plane}.
\newblock {\em Geombinatorics}, 29(4):137, 2020.

\bibitem{Parts2020_5}
J.~Parts.
\newblock {The chromatic number of the plane is at least 5 a human-verifiable
  proof}.
\newblock {\em Geombinatorics}, 29(2):77--102, 2020.

\bibitem{Parts2020_percent}
Jaan Parts.
\newblock What percent of the plane can be properly 5- and 6-colored?, 2020.

\bibitem{Payne2009}
Michael~Stuart Payne.
\newblock Unit distance graphs with ambiguous chromatic number.
\newblock {\em Electronic Journal of Combinatorics}, 16(1):1--7, 2009.

\bibitem{Perz2018}
Daniel Perz.
\newblock {Triangles in the colored Euclidean plane}.
\newblock Master's thesis, Graz University of Technology, 2018.

\bibitem{Polymath16}
Polymath16.
\newblock
  \url{https://asone.ai/polymath/index.php?title=Hadwiger-Nelson_problem}.
\newblock Polymath Project.

\bibitem{Pritikin1998}
Dan Pritikin.
\newblock All unit-distance graphs of order 6197 are 6-colorable.
\newblock {\em J. Comb. Theory, Ser. {B}}, 73(2):159--163, 1998.

\bibitem{RuzsaTuzaVoigt2002}
I.Z. Ruzsa, Zs. Tuza, and M.~Voigt.
\newblock Distance graphs with finite chromatic number.
\newblock {\em Journal of Combinatorial Theory, Series B}, 85(1):181 -- 187,
  2002.

\bibitem{ScheinermanUllman2011}
Edward~R Scheinerman and Daniel~H Ullman.
\newblock {\em Fractional graph theory: a rational approach to the theory of
  graphs}.
\newblock Courier Corporation, 2011.

\bibitem{ShelahSoifer2003}
Saharon Shelah and Alexander Soifer.
\newblock Axiom of choice and chromatic number of the plane.
\newblock {\em J. Comb. Theory Ser. A}, 103(2):387--391, 2003.

\bibitem{ShelahSoifer2004}
Saharon Shelah and Alexander Soifer.
\newblock Axiom of choice and chromatic number: examples on the plane.
\newblock {\em Journal of Combinatorial Theory, Series A}, 105(2):359 -- 364,
  2004.

\bibitem{Soifer2005}
Alexander Soifer.
\newblock Axiom of choice and chromatic number of $\mathbb{R}^n$.
\newblock {\em Journal of Combinatorial Theory, Series A}, 110(1):169 -- 173,
  2005.

\bibitem{Soifer2009}
Alexander Soifer.
\newblock {\em The Mathematical Coloring Book: Mathematics of Coloring and the
  Colorful Life of its Creators}.
\newblock Springer-Verlag New York, 2009.

\bibitem{Steinhardt2009}
Jacob Steinhardt.
\newblock On coloring the odd-distance graph.
\newblock {\em The Electronic Jounal of Combinatorics}, 16:N12, 2009.

\bibitem{Townsend1979}
S.~P. Townsend.
\newblock Every 5-colouring map in the plane contains a monochrome unit.
\newblock {\em Journal of Combinatorial Theory, Series A}, 30:114 -- 115, 1979.

\bibitem{Townsend2005}
S.~P. Townsend.
\newblock Colouring the plane with no monochrome unit.
\newblock {\em Geombinatorics}, XIV:181 -- 193, 2005.

\bibitem{VoronovNeopryatnayaDergachev2021}
Vsevolod Voronov, Anna Neopryatnaya, and Eugene Dergachev.
\newblock Constructing 5-chromatic unit distance graphs embedded in the
  euclidean plane and two-dimensional spheres, 2021.

\bibitem{WalczakWojciechowski2006}
Zbigniew Walczak and Jacek~M. Wojciechowski.
\newblock Transmission scheduling in packet radio networks using graph coloring
  algorithm.
\newblock In {\em 2006 International Conference on Wireless and Mobile
  Communications (ICWMC'06)}, pages 46--46, 2006.

\bibitem{WilliamsYan2001}
H.~P. Williams and Hong Yan.
\newblock Representations of the all\_different predicate of constraint
  satisfaction in integer programming.
\newblock {\em INFORMS J. on Computing}, 13(2):96--103, 2001.

\bibitem{WOODALL1973}
D.~R Woodall.
\newblock Distances realized by sets covering the plane.
\newblock {\em Journal of Combinatorial Theory, Series A}, 14:187 -- 200, 1973.

\end{thebibliography}
\end{document}